\documentclass[11pt,leqno,oneside]{amsart}
\usepackage{amsmath,amssymb,,comment,graphicx,psfrag}
\usepackage{bm,mathrsfs}
\usepackage{mathtools}
\usepackage{color}
\usepackage{enumerate}
\usepackage{pdfsync}
\usepackage{hhline}
\usepackage[margin=1in]{geometry}
\usepackage{multicol}
\usepackage{multirow}
\usepackage{booktabs}
\usepackage[colorlinks,linktocpage,linkcolor=blue]{hyperref}
\usepackage{placeins}
\usepackage{subfigure}
\usepackage{float}
\allowdisplaybreaks

\def\baselinestretch{1.0}
\newcommand{\al}{\alpha}
\newcommand{\fy}{\varphi}

\def\Dal{{\partial_t^\al}}

\def\d{{\rm d}}

\theoremstyle{plain}% default
\newtheorem{theorem}{Theorem}[section]
\newtheorem{proposition}{Proposition}[section]
\newtheorem{remark}{Remark}[section]

\newtheorem{lemma}{Lemma}[section]
\newtheorem{ass}{Assumption}[section]

\newtheorem{example}{Example}[section]

\newtheorem{corollary}{Corollary}[section]
\numberwithin{equation}{section}

\usepackage{titlesec}
\titleformat{\section}{\vskip10pt\large\bfseries}{\thesection.}{0.5em}{\centering\vspace{5pt}}
\titleformat{\subsection}{\vskip10pt\normalsize\bfseries}{\thesubsection.}{0.5em}{}
\newcommand{\R}{\mathbb{R}}

\def\d{{\rm d}}
\def\al{\alpha}

\def\Dal{{\partial^\alpha_t}}

\def\dH#1{\dot H^{#1}(\Omega)}

\def\II{(\Omega)}

\graphicspath{{./figures/}}

\begin{document}

\title[High-order splitting FEMs for subdiffusion]{High-order splitting finite element methods for the \\
subdiffusion equation with limited smoothing property}

\author[Buyang Li]{$\,\,$Buyang Li$\,$}
\address{Department of Applied Mathematics,
The Hong Kong Polytechnic University, Kowloon, Hong Kong}
\email {bygli@polyu.edu.hk}

\author[Zongze Yang]{$\,\,$Zongze Yang$\,$}
\address{Department of Applied Mathematics,
The Hong Kong Polytechnic University, Kowloon, Hong Kong}
\email {zongze.yang@polyu.edu.hk}

\author[Zhi Zhou]{$\,\,$Zhi Zhou$\,$}
\address{Department of Applied Mathematics,
The Hong Kong Polytechnic University, Kowloon, Hong Kong}
\email {zhizhou@polyu.edu.hk}


\begin{abstract}
In contrast with the diffusion equation which smoothens the initial data to $C^\infty$ for $t>0$ (away from the corners/edges of the domain), the subdiffusion equation only exhibits limited spatial regularity. As a result, one generally cannot expect high-order accuracy in space in solving the subdiffusion equation with nonsmooth initial data. In this paper, a new splitting of the solution is constructed for high-order finite element approximations to the subdiffusion equation with nonsmooth initial data. The method is constructed by splitting the solution into two parts, i.e., a time-dependent smooth part and a time-independent nonsmooth part, and then approximating the two parts via different strategies. The time-dependent smooth part is approximated by using high-order finite element method in space and convolution quadrature in time, while the steady nonsmooth part could be approximated by using smaller mesh size or other methods that could yield high-order accuracy. Several examples are presented to show how to accurately approximate the steady nonsmooth part, including piecewise smooth initial data, Dirac--Delta point initial data, and Dirac measure concentrated on an interface.
The argument could be directly extended to subdiffusion equations with nonsmooth source data. Extensive numerical experiments are presented to support the theoretical analysis and to illustrate the performance of the proposed high-order splitting finite element methods.
\end{abstract}

\maketitle

\renewcommand{\baselinestretch}{0.95}
\setlength\abovedisplayskip{4.0pt}
\setlength\belowdisplayskip{4.0pt}

\section{Introduction}\label{sec:intro}
This article is concerned with the construction and analysis of high-order finite element methods for solving the subdiffusion equation in a convex polygonal/polyhedral domain $\Omega\subset\mathbb{R}^d$, with $d\ge 1$, i.e.,
\begin{align}\label{eqn:fde}
\left\{\begin{aligned}
\partial_t^\alpha u - \Delta u &=f
&&\mbox{in}\,\,\,\Omega\times(0,T) ,\\
u &=0 &&\mbox{on}\,\,\,\partial \Omega\times(0,T),\\
u(0)&=u^0  &&\mbox{in}\,\,\,\Omega,\\
\end{aligned}\right.
\end{align}
where $f$ and $u^0$ are given source function and initial value, respectively,
$\Delta: H^2(\Omega)\cap H^1_0(\Omega)\rightarrow L^2(\Omega)$ is the Dirichlet Laplacian operator,
%and $A$ is a second-order elliptic partial differential operator.
and $\Dal u$ denotes the Djrbashian--Caputo fractional time derivative of order
$\alpha\in(0,1)$  \cite[p. 92]{KilbasSrivastavaTrujillo:2006} and \cite[Section 2.3]{Jin:2021book}:
\begin{equation*}
  \Dal u (t) = \frac{1}{\Gamma(1-\alpha)}\int_0^t (t-s)^{-\alpha}  u'(s)\ {\rm d}s
  \quad\mbox{with}\quad \Gamma(1-\alpha)=\int_0^\infty s^{-\alpha}e^{-s}\d s ;
\end{equation*}

%The solution of the subdiffusion problem in \eqref{eqn:fde} generally has limited regularity even for smooth initial data and source terms.
%In this article, we focus on the case of low regularity data,
%which is typical for related inverse problems \cite{ZhangZhou:2020} and optimal control problems
%\cite{JinLiZhou:2020ima}.

The subdiffusion equation in \eqref{eqn:fde} has received much attention in recent years in physics, engineering,
biology and finance, due to their capability for describing anomalously slow diffusion
processes, also known as subdiffusion.
At a microscopic level, subdiffusive processes can be
described by continuous time random walk with a heavy-tailed waiting time distribution,
which displays local motion occasionally interrupted by
long sojourns and trapping effects. These transport processes are characterized by a sublinear
growth of the mean squared displacement of the particle with the time, as opposed to linear
growth for Brownian motion. The model \eqref{eqn:fde} has found many successful practical
applications,
e.g., subsurface flows \cite{AdamsGelhar:1992}, thermal diffusion in
media with fractal geometry \cite{Nigmatulli:1986}, transport column experiments \cite{HatanoHatano:1998}
and heat conduction with memory \cite{Wolfersdorf:1994}, to name but a few.
See \cite{MetzlerKlafter:2000} for physical modeling and a long list of applications.

One of the main difficulties in the numerical approximation to the subdiffusion equation, compared to the standard parabolic equations, is the weak singularity of the solution at $t=0$ and the limited regularity pick up with respect to the initial data. In general, for the subdiffusion equation with a nonsmooth initial value $u^0\in L^2(\Omega)$ and a temporally smooth source function $ f(x,t)$, the solution generally exhibits the following type of weak singularity at $t=0$:
\begin{align}\label{regulairty-sigma}
\|\partial_t^m u(\cdot,t)\|_{L^2} \le C_m t^{-m}
\quad\mbox{for}\,\,\, m\ge 0 ,
\end{align}
and the spatial regularity pick up is limited to
\begin{align}\label{regulairty-pick-up}
\|u(\cdot,t)\|_{H^2} \le Ct^{-1}  .
\end{align}
Higher-order spatial regularity generally cannot be expected for $u^0\in L^2(\Omega)$.
This limited smoothing property was shown in the paper \cite{SakamotoYamamoto:2011}
of Sakamoto and Yamamoto with the following two-sided stability (with $f\equiv0$):
\begin{equation}\label{eqn:reg-lim}
c_1\| u^0  \|_{\dH s} \le \|  u(T)  \|_{\dH {s+2}} \le c_2 \| u^0   \|_{\dH s}.
\end{equation}
The limited regularity of the solution, as shown in \eqref{regulairty-sigma}--\eqref{eqn:reg-lim},
causes many difficulties in developing high-order temporal and spatial discretizations for the subdiffusion equation when the initial data is nonsmooth.
Many efforts have been made in overcoming these difficulties. In particular, high-order temporal discretizations for the subdiffusion equation have been developed based on graded mesh in \cite{Kopteva-2021,Kopteva-Meng-2020,Mustapha-McLean-2011,Stynes-2017} for $u^0\in H^1_0(\Omega)\cap H^2(\Omega)$, discontinuos Galerkin method in \cite{Mustapha-2014} for $u^0\in H^{5/2}(\Omega)\cap H^1_0(\Omega)$ plus a compatibility condition $\Delta u^0=0 $ on $\partial\Omega$,  BDF convolution quadrature in \cite{JinLiZhou:2017sisc} and Runge--Kutta convolution quadrature in \cite{Banjai2019} for $u^0\in L^2(\Omega)$, and exponential convolution quadrature in \cite{LM22Exp} for semilinear problems with $u^0\in L^\infty(\Omega)$.
See also \cite{BanjaiMakridakis:2022} for \textsl{a posteriori} error analysis of several popular time stepping schemes.
%However, the development of high-order spatial discretizations for such nonsmooth initial data still remains challenging and missing from the literature.

The spatial discretization using the standard Galerkin finite element method (FEM) or
the lumped mass Galerkin FEM for solving the subdiffusion equation with nonsmooth initial data was studied in \cite{JinLazarovZhou:2013sinum}.
The second order convergence in $L^2\II$ norm was established and it is optimal with respect to the $L^2\II$ initial data.
In these works, the error analysis was carried using the Mittag--Leffler functions.
Error estimates of the standard Galerkin FEM with initial data in $\dot H^q(\Omega)$ with $q\in(-1,0)$ can be found in \cite{JinLazarovPasiackZhou:2013naa}.
The Laplace transform approach, initially introduced for parabolic equations by Fujita and Suzuki \cite[Chapter 7]{FujitaSuzuki:1991}, was adapted to the subdiffusion equation in \cite[Chapter 2.3]{JinZhou:book} to remove a logarithmic factor in the error estimates. See \cite{JinLazarovZhou:2019cmame} and \cite[Chapter 2]{JinZhou:book}
for concise overviews. The energy argument,  which represents one of the most commonly used strategies for standard parabolic equations, is much more involved for the subdiffusion equation.
This is due to the nonlocality of the  fractional derivative $\partial_t^\alpha u$, which causes that many useful tools
(such as integration by parts formula and product rule) are either invalid or requiring substantial modifications.
Some first encouraging theoretical results in this important direction were obtained by Mustapha \cite{Mustapha:2018mc}, where optimal error estimates for the homogeneous problem were obtained using an energy argument; see also \cite{LeMcLeanMustapha:2016} for the time-fractional Fokker--Planck equation. A unified analysis of different kinds of FEMs for the homogeneous subdiffusion problem based on an energy argument, which generalizes the corresponding technique
for standard parabolic problems in \cite[Chapter 3]{Thomee:2006}, were given by Karaa \cite{Karaa:2017}.
The numerical analysis of subdiffusion equations with irregular domains and finite volume element methods were discussed in \cite{LeMcLeanLamichhane:2017} and \cite{KaraaMustaphaPani:2017,KaraaPani:2018}, respectively.

%See also \cite{KaraaMustapha:2017} for a novel energy argument of FEM. % and \cite{Pani:2017} for a finite volume method.
In all the aforementioned work for the subdiffusion equation, people only considered the error analysis of the piecewise linear finite element method for nonsmooth initial data.
This is attributed to the limited smoothing property of the subdiffusion equation, which only smoothen the initial data $u^0\in L^2(\Omega)$ to $u(t)\in H^2(\Omega)$ for $t>0$; cf. \eqref{eqn:reg-lim} with $s=0$.
This is in sharp contrast to the standard parabolic equation, which smoothen the initial data $u^0\in L^2(\Omega)$ to $u(t)\in C^\infty(\Omega)$ for $t>0$.
%For example, if the initial data only belongs to $L^2\II$, the solution  cannot be better than an $H^2\II\cap H_0^1\II$ function.
Consequently, the standard finite element approximation to the subdiffusion equation only has second-order convergence in $L^2(\Omega)$ for initial data in $L^2(\Omega)$.
%Specifically, let $u$ be the exact solution to
%the subdiffusion model \eqref{eqn:fde} with $f\equiv0$ and $u_h$ be the conforming FEM solution using piecewise $q$-th order polynomials with $q\ge 1$. Then for
%nonsmooth initial data $u^0 \in \dH s$ with $s\in[-1,0]$ there holds \cite[Theorem 3.1]{JinLazarovPasiackZhou:2013naa}
%\begin{equation*}
% \|  u(t) - u_h(t)   \|_{L^2\II} \le c h^{2+s} t^{-\alpha} \|  u^0 \|_{\dH s} \quad \text{for all}~~ t\in(0,T].
%\end{equation*}
%However, for the parabolic counterpart (i.e. $\alpha=1$), there holds \cite[Theorem 3.3]{Thomee:2006}
%\begin{equation*}
% \|  u(t) - u_h(t)   \|_{L^2\II} \le c h^{q+1} t^{-(q+1-s)/2} \|  u^0 \|_{\dH s} \quad \text{for all}~~ t\in(0,T].
%\end{equation*}
%This phenomenon is supported by the numerical experiments. For example, Table \ref{table:exa52_space_m0}
%presents the empirical convergence rates for FEM approximation when the initial data is an almost $\dH{-1}$ function.
%For the parabolic equation (i.e. $\alpha=1$) the convergence is of order $O(h^{q+1})$, while for the fractional case ($\alpha=0.6$)
%the convergence is always first-order.
The development of high-order finite element methods in space in case of nonsmooth initial data remains challenging and missing in the literature.
If the problem data is smooth  and compatible to the boundary condition, then the solution could be smooth enough. In this case, it is possible to construct high-order
spatial approximation. For example, the
high-order hybridizable discontinuous Galerkin method was proposed and analyzed in \cite{CockburnMustapha:2015}
under the assumption that the solution is smooth enough.

In this paper, we construct a splitting method which allows to develop high-order finite element methods for solving the subdiffusion equation with nonsmooth initial data.
For $u^0 \in L^2\II$, we split the solution into a time-dependent regular part $u^r(t) \in \dH {2m+2} $ and
several time-independent singular parts $u_j^s$ with $j=1,\ldots,m$.
Then the time-dependent smooth part $u^r(t)$ is approximated by using high-order finite element methods in space and convolution quadrature in time generated by $k$-step BDF method, with $k=1,2,\ldots,6$. If we denote by $U_h^{r,n}$ the fully discrete solution approximating $u^r(t_n)$, then the following result is proved (see Theorems \ref{thm:regular} and \ref{thm:fully-err}):
$$  \| U_h^{r,n} - u^r(t_n) \|_{L^2\II} \le c \Big(h^{2m+2} t_n^{-(1+m)\alpha} + \tau^k t_n^{-k-m\alpha}\Big) \| u^0 \|_{L^2\II} \quad \text{for all}~~ t\in(0,T].$$
Note that the integer $m$ could be arbitrarily large, and hence we have developed an arbitrarily high-order FEM approximation for the smooth part.
This argument also works for weaker initial data $u^0 \in \dH s$ with $s<0$.
Meanwhile, the singular parts $u_j^s$ are independent of time and they can be approximated by solving several elliptic equations with nonsmooth sources.
This is illustrated in Section \ref{sec:singular} for several exemplary nonsmooth data, e.g., piecewise smooth functions, Dirac--Delta point source, and Dirac measure concentrated on an interface.
More generally, the time-independent nonsmooth part can be approximated by the standard FEM using smaller mesh size without increasing the overall computational cost significantly.
This is possible as the nonsmooth part is time-independent and therefore avoids the time stepping procedure.
As a result, the high-order finite element approximation to the subdiffusion equation can be realized by the novel splitting strategy. %the  largely comparable to that for the parabolic equation \cite[Chapter 3]{Thomee:2006}.
This strategy  works for all second-order elliptic operators with smooth coefficients even though we only consider the negative Laplacian $-\Delta$ for simplicity of presentation.
As far as we know, this is the first attempt to develop spatially high-order finite element methods for the subdiffusion equation with nonsmooth initial data.
In addition, the argument in this paper can be easily extended to the case of nonsmooth source term.

The rest of the paper is organized as follows.
In Section \ref{sec:reg}, we present the splitting method and the high-order finite element approximation to the regular part of the solution.
High-order time-stepping schemes for the regular part and the corresponding error estimates are presented in Section \ref{sec:time}.
High-order finite element approximations to the singular parts are discussed in Section \ref{sec:singular}.
The extension to nonsmooth source term is discussed in Section \ref{sec:singular}.
Finally, in Section \ref{sec:numerics}, we present several numerical examples to illustrate the high-order convergence 
of the proposed splitting FEMs in comparison with the standard Galerkin FEMs for the subdiffusion equation.
% two-dimensional numerical results to complement the analysis.

\section{Construction of high-order spatial discretizations}\label{sec:reg}
Let $\{\lambda_j\}_{j=1}^\infty$ and $\{\varphi_j\}_{j=1}^\infty$ be the eigenvalues (ordered nondecreasingly with multiplicity counted) and the $L^2(\Omega)$-orthonormal eigenfunctions, respectively, of the elliptic operator $A=-\Delta:H^2(\Omega)\cap H^1_0(\Omega)\rightarrow L^2(\Omega)$ under the zero boundary condition. Then $\{\varphi_j\}_{j=1}^\infty$
forms an orthonormal basis in $L^2(\Omega)$.
For any real number $s\ge 0$, we denote by $\dH s$ the Hilbert space
with the induced norm $\|\cdot\|_{\dH s}$ defined by
\begin{equation*}
  \|v\|_{\dH s}^2 :=\sum_{j=1}^{\infty}\lambda_j^s\langle v,\varphi_j \rangle^2.
\end{equation*}
In particular, $\|v\|_{\dH 0}=\|v\|_{L^2(\Omega)}=(v,v)^\frac{1}{2}$ is the norm in $L^2(\Omega)$.
For $s< 0$, we define the space $\dH s = \dH {-s}'$.
It is straightforward to verify that $\|v\|_{\dH 1} = \|\nabla v\|_{L^2(\Omega)}$ is an equivalent norm of $H_0^1(\Omega)$
and  $\|v\|_{\dH 2}=\|A v\|_{L^2(\Omega)}$ is a norm of $H^2(\Omega)\cap H^1_0(\Omega)$; see \cite[Section 3.1]{Thomee:2006}.
Moreover, for any integer $m\ge 0$, $v\in \dot H^{2m+2}(\Omega)$ if and only if
$A^{j-1} v\in H^2(\Omega)\cap H^1_0(\Omega)$ and $A^jv\in \dot H^{2m+2-2j}(\Omega)$ for $j=1,\dots,m+1$, and
$$
\|v\|_{\dot H^{2m+2}(\Omega)}
\sim \|A^{m+1}v\|_{L^2(\Omega)} \quad\forall\, v\in \dot H^{2m+2}(\Omega) .
$$

%There are several ways to represent the solution to the subdiffusion problem in \eqref{eqn:fde}.
%One popular way is to apply the Laplace transform to the subdiffusion equation and then, after obtaining an expression for
%the Laplace transform of the solution, apply the inverse Laplace transform to represent the solution.
It is known that the solution to problem \eqref{eqn:fde} can be written as (cf. \cite[eq. (6.24)]{Jin:2021book})
\begin{align}\label{eqn:sol-1}
u(t)= F(t)v + \int_0^t E(t-s) f(s) \d s ,
\end{align}
where the solution operators $F(t)$ and $E(t)$ are defined by
\begin{align}\label{eqn:sol-2}
F(t):=\frac{1}{2\pi {\rm i}}\int_{\Gamma_{\theta,\kappa}}e^{zt} z^{\alpha-1} (z^\alpha +A  )^{-1}\, \d z \quad\mbox{and}\quad
E(t):=\frac{1}{2\pi {\rm i}}\int_{\Gamma_{\theta,\kappa}}e^{zt}  (z^\alpha+A )^{-1}\, \d z ,
\end{align}
respectively,
with integration over a contour $\Gamma_{\theta,\kappa}$ in the complex plane $\mathbb{C}$
(oriented counterclockwise), defined by
$\Gamma_{\theta,\kappa}=\left\{z\in \mathbb{C}: |z|=\delta, |\arg z|\le \theta\right\}\cup
  \{z\in \mathbb{C}: z=\rho e^{\pm\mathrm{i}\theta}, \rho\ge \kappa\} .$
Throughout, we fix $\theta \in(\frac{\pi}{2},\pi)$ so that $z^{\al} \in \Sigma_{\al\theta}
\subset \Sigma_{\theta}:=\{0\neq z\in\mathbb{C}: |{\rm arg}(z)|\leq\theta\}$ for all $z\in\Sigma_{\theta}$.
The following resolvent estimat will be frequently used (see
\cite[Example 3.7.5 and Theorem 3.7.11]{ArendtBattyHieber:2011}):
\begin{equation} \label{eqn:resol}
  \| (z+A )^{-1} \|\le c_\theta |z|^{-1},  \quad \forall z \in \Sigma_{\theta},
  \,\,\,\forall\,\theta\in(0,\pi),
\end{equation}
where $\|\cdot\|$ denotes the operator norm on $L^2(\Omega)$.

%{\color{blue}(Is the following paragraph useful for this paper? Otherwise we can remove it to save the pages? In this case, we can merge Sections 2 and 3.)}
Equivalently, %using the eigenfunction expansion $\{(\lambda_j,\varphi_j)\}_{j=1}^\infty$,
the solution operators in \eqref{eqn:sol-2} can also be expressed as
\begin{equation*}
  F(t)v = \sum_{j=1}^\infty E_{\alpha,1}(-\lambda_jt^\alpha)(v,\varphi_j)\varphi_j\quad \mbox{and}\quad
  E(t)v = \sum_{j=1}^\infty t^{\alpha-1}E_{\alpha,\alpha}(-\lambda_jt^\alpha)(v,\varphi_j)\varphi_j.
\end{equation*}
Here $E_{\alpha,\beta}(z)$ is the two-parameter Mittag--Leffler function
\cite[Section 1.8, pp. 40-45]{KilbasSrivastavaTrujillo:2006}.
The Mittag--Leffler function $E_{\alpha,\beta}(z)$ is a generalization of the exponential function
$e^z$ appearing in normal diffusion. For any $\alpha\in (0,1)$, the function $E_{\alpha,1}
(-\lambda t^\alpha)$ decays only polynomially like $\lambda^{-1}t^{-\alpha}$ as $\lambda,t\to\infty$ \cite[equation (1.8.28),
p. 43]{KilbasSrivastavaTrujillo:2006}, which contrasts sharply with the exponential decay for $e^{-\lambda t}$
appearing in normal diffusion. This important feature directly translates into the limited smoothing
property in space for the solution operators $F(t)$ \cite{SakamotoYamamoto:2011} .
As a result, in general, we cannot expect high-order approximations for subdiffusion problem \eqref{eqn:fde}
with nonsmooth data. In fact, our preceding study \cite{JinLazarovPasiackZhou:2013naa}
shows an optimal convergence rate $O(h^{2-s})$ for piecewise linear finite element approximation, when $u^0 \in \dH {-s} $.

%Our objective is to construct a high-order finite element method for the subdiffusion equation based on a new splitting of the solution into regular and singular parts, separately.
%The regular part corresponds to the solution of an evolution problem with a smooth initial value, and the singular part corresponds to the solution of an elliptic problem with a nonsmooth source term.
%Then different methods are used to approximate the two parts to high-order accuracy.

\subsection{A new splitting of the solution}
For the simplicity of presentation, we consider the homogeneous subdiffusion equation with $f\equiv 0$ and $u^0 \in L^2\II$.
The argument could be extended to the weaker case $u^0 \in H^s$ with $s\in [-1,0)$ by slight modifications.
The case with nonsmooth $f\ne 0$ will be discussed in Section \ref{sec:singular}.

We split the integrand of \eqref{eqn:sol-2} into two parts based on the following relation:
\begin{align}\label{identity}
(z^\alpha+A)^{-1} = A^{-1} -  z^{\alpha}(z^\alpha+A)^{-1}A^{-1} .
\end{align}
This leads to the following splitting of the solution operator:
\begin{equation*}
  \begin{aligned}
  F(t):=\frac{1}{2\pi {\rm i}}\int_{\Gamma_{\theta,\kappa }}e^{zt} z^{\alpha-1} (z^\alpha +A  )^{-1}\, \d z
       =&\frac{1}{2\pi {\rm i}}\int_{\Gamma_{\theta,\kappa }}e^{zt} \Big[z^{\alpha-1}A^{-1} -  z^{2\alpha-1} (z^\alpha +A )^{-1}A^{-1}\Big]\, \d z \\
       =&\frac{t^{-\alpha}}{\Gamma(1-\alpha)} A^{-1} - \frac{1}{2\pi {\rm i}}\int_{\Gamma_{\theta,\kappa }}e^{zt} z^{2\alpha-1} (z^\alpha +A )^{-1}A^{-1}\, \d z .
  \end{aligned}
\end{equation*}
In the case $f\equiv 0$ we obtain the following splitting of the solution:
\begin{equation}\label{eqn:split-1}
  \begin{aligned}
    u(t) %&= \frac{1}{2\pi\rm{i}} \int_\Gamma e^{zt} z^{\alpha-1} (z^\alpha+A)^{-1} v\,dz\\
    = F(t) u^0 &= u^s + u^r(t)
\end{aligned}
\end{equation}
with
\begin{equation*}
  \begin{aligned}
   u^s  = \frac{1}{\Gamma(1-\alpha)} A^{-1} t^{-\alpha} u^0\quad \text{and}\quad u^r(t) =  -  \frac{1}{2\pi\rm{i}} \int_{\Gamma_{\theta,\kappa}} e^{zt} z^{2\alpha-1} (z^\alpha+A)^{-1}A^{-1} u^0\,\d z ,
\end{aligned}
\end{equation*}
denoting the singular and regular parts in this splitting, respectively.
%The regular component $u^r(t)$ will be directly approximated by high-order finite element methods,
%while the singular component $u^s$ is a solution of an elliptic equation with an $L^2\II$ source data. The latter will be approximated separately by different methods.
This splitting process can be continued by substituting relation \eqref{identity} into the expression of the regular part repeatedly. Then we can obtain the following higher-order splitting:
\begin{equation}\label{eqn:split-2}
  \begin{aligned}
    u(t) =  u^r(t) + \sum_{j=1}^m u_j^s
\end{aligned}
\end{equation}
with% $u_j^s$ and $u^r$ are respectively defined as
\begin{equation}\label{eqn:split-3}
  \begin{aligned}
    &u_j^s  = (-1)^{j+1} \frac{t^{-j\alpha}}{\Gamma(1-j\alpha)} A^{-j} u^0,\quad \text{for}~~j=1,2,\ldots,m, \\[5pt]
    &u^r(t) = F^r(t) u^0 := \frac{(-1)^{m}}{2\pi\rm{i}} \int_{\Gamma_{\theta,\kappa}} e^{zt} z^{(1+m)\alpha-1} (z^\alpha+A)^{-1}A^{-m} u^0 \,\d z.
\end{aligned}
\end{equation}
Note that the singular part $u_j^s(t)$ is a solution of an elliptic problem, while the regular part $u^r(t)$ corresponds to the solution of a non-standard evolution problem with a relatively smooth initial value $A^{-m} u^0\in \dH{2m}$. Since $(z^\alpha+A)^{-1}$ maps $\dH{2m}$ to $\dH{2m+2}$, it follows that the regular part $u^r(t)$ is in $\dH{2m+2}$ and
\begin{equation}\label{eqn:reg-sm-est}
  \| u^r(t) \|_{\dH{2m+2}} =  \|F^r(t) u^0\|_{\dH{2m+2}} \le c t^{-(m+1)\alpha} \| u^0 \|_{L^2\II}
\end{equation}

In the next two subsections, we present error estimates for the Lagrange interpolation and the Ritz projection of functions in $\dH{2m+2}$, and then use the established results to prove high-order convergence of the finite element approximation to the regular part $u^r(t)$.
The approximation to the singular part will be discussed in Section \ref{sec:singular}.

\subsection{Finite element approximations to functions in $\dH{2m+2}$}\label{section:FEM}

%In this subsection, we introduce a numerical scheme for approximating the regular part $u^r(t)$ in \eqref{eqn:split-1}
%using high-order finite element method.
We assume that the polygonal domain $\Omega$ is partitioned into a set $\mathcal{K}_h$ of shape-regular and locally quasi-uniform triangles with mesh size $h=\max_{K\in\mathcal{K}_h}{\rm diam}(K)$, and denote by $X_h$ the Lagrange finite element space of degree $2m+1$ subject to the partition, i.e.,
$$
X_h = \{ v_h \in H_0^1\II :
v_h |_K \in \mathbb{P}_{2m+1}~\text{for all}~K \in \mathcal{K}_h   \} ,
$$
where $\mathbb{P}_{2m+1}$ denotes the space of polynomials of degree $2m+1$.

Let $P_h:L^2(\Omega)\rightarrow X_h$, $R_h:\dH1\to X_h$ and $A_h: X_h\rightarrow X_h$ be the $L^2$-orthogonal projection, the Ritz projection operator, and the discrete elliptic operator, respectively, defined by
\begin{align*}
\begin{aligned}
& (P_h\psi,\chi)=(\psi,\chi)&&\forall\psi\in L^2(\Omega),\,\chi\in X_h,\\
&(\nabla R_h \psi,\nabla\chi)  =(\nabla \psi,\nabla\chi) && \forall \psi\in \dot H^1(\Omega),\,\chi\in X_h,\\
& (A_h\psi,\chi)=(\nabla\psi,\nabla\chi)&&\forall\psi,\,\chi\in X_h.
\end{aligned}
\end{align*}
We shall work with the following assumption on the triangulation of the domain.

\begin{ass}\label{ass:mesh}
We assume that the triangulation is locally refined towards the corners and edges of the domain,
such that both the Lagrange interpolation $I_h:C(\overline\Omega)\rightarrow X_h$
and the Ritz projection $R_h:H^1_0(\Omega)\rightarrow X_h$
%finite element solution to the Poisson equation
have optimal-order convergence, i.e.,
\begin{align}
\|v-I_hv\|_{L^2(\Omega)}  +  h\|v-I_hv\|_{H^1(\Omega)}
&\leq ch^{2r+2}\|v\|_{\dot H^{2r+2}(\Omega)}
\label{eqn:int-err}\\
\|v-R_hv\|_{L^2(\Omega)}  +  h\|v-R_hv\|_{H^1(\Omega)}
&\leq ch^{2r+2}\|v\|_{\dot H^{2r+2}(\Omega)}
\label{Error-solution-f}
\end{align}
for $v\in \dot H^{2r+2}(\Omega) $ and $0\le r\le m$.
\end{ass}

%\begin{example}\upshape
For example, in a two-dimensional polygonal domain $\Omega$, Assumption \ref{ass:mesh} is satisfied by the following type of graded mesh (see Proposition \ref{proposition:FEM} in Appendix):
\begin{align}\label{mesh-condition}
\hbar(x)\sim
\left\{\begin{aligned}
&|x-x_0|^{1-\gamma} h \quad\mbox{with}\,\,\,\gamma\in \Big(0, \frac{\min(1,\pi/\theta) }{2m+1}  \Big)
&&\mbox{for}\,\,\,  h_*\le |x-x_0|\le d_0 \\
&h_* &&\mbox{for}\,\,\, |x-x_0|\le h_* ,
\end{aligned}\right.
\end{align}
where $\hbar(x)$ denotes the spatially dependent diameter of triangles,
$x_0$ is a corner of the polygon with interior angle $\theta$,  $h_*\sim h^{1/\gamma}$, and $d_0$ is a constant such that $D_0'=\{x\in\Omega:|x-x_0|<2d_0\}$ is a sector centred at the corner $x_0$.

As a direct consequence of the approximation property \eqref{eqn:int-err}--\eqref{Error-solution-f}, we have the following estimate in the negative Sobolev spaces.

%For an elliptic equation $A \phi = f =A^{-j} \psi$ with some $\psi\in L^2\II$ and $0\le j\le m$, the corresponding finite element method is to find $\phi_h \in X_h$ such that $A_h \phi_h = P_h f$, i.e., $\phi_h = R_h \phi$.   With locally refined meshes, we have the following error estimate for the
%finite element approximation: \red{maybe write  a small lemma for \eqref{Error-solution-f} and \eqref{eqn:err-ellip-1}???}
%\begin{align}\label{Error-solution-f}
%  \| R_h \phi - \phi \|_{L^2\II} + h \| R_h \phi - \phi \|_{H^1\II}
%  \leq ch^{2j+2}\| f \|_{\dH{2j}}
%  \leq ch^{2j+2}\| \psi\|_{L^2(\Omega)}.
%\end{align}
%This together with the duality argument immediately implies the following estimate with negative norms:
%

\begin{lemma} \label{eqn:err-ellip-1}
Under Assumption \ref{ass:mesh} with $m\ge 1$, the following estimate holds %for $\phi \in \dH1$ {\color{blue}(in the case $m\ge 1$)}:
\begin{align*}
  \| R_h \phi - \phi \|_{\dH{-r}}\le  ch^{r+1}\| R_h \phi - \phi\|_{\dH{1}}, \quad \text{with}~ ~ 1\le r \le 2m.
\end{align*}
\end{lemma}

\begin{proof}
For any $\psi\in \dH r$ we let $w = A^{-1} \psi \in \dH{r+2}$. By the duality argument, we have
 \begin{align*}
  \| R_h \phi - \phi \|_{\dH{-r}} &= \sup_{\psi\in \dH r} \frac{\langle R_h \phi - \phi,\psi \rangle}{\|  \psi  \|_{\dH r}}
  = \sup_{\psi\in \dH r} \frac{  (\nabla(R_h \phi - \phi),\nabla w)}{\|  \psi  \|_{\dH r}}\\
  &=\sup_{\psi\in \dH r} \frac{ (\nabla (R_h \phi - \phi),\nabla(w-I_h w))}{\|  \psi  \|_{\dH r}}.
\end{align*}
Then using \eqref{eqn:int-err} we derive
 \begin{align*}
 | (\nabla (R_h \phi - \phi),\nabla(w-I_h w)) |& \le \|  R_h \phi - \phi \|_{\dH 1} \|  w-I_h w \|_{\dH 1}  \\
 &\le   ch^{r+1}\| R_h \phi - \phi\|_{\dH{1}}\|  w \|_{\dH {r+2}}\\
 &= ch^{r+1}\| R_h \phi - \phi\|_{\dH{1}} \|  \psi \|_{\dH r}.
\end{align*}
Then the desired result follows immediately.
\end{proof}

\subsection{High-order approximation to the regular part $u^r(t)$}\label{section:regular}

In order to approximate the regular part $u^r(t)$ in \eqref{eqn:split-1},
we consider the following contour integral
\begin{align}\label{eqn:semi-reg}
u_h^r(t) =  F_h^r(t) P_h u^0 := \frac{(-1)^m}{2\pi\rm{i}} \int_{\Gamma_{\theta,\kappa}} e^{zt} z^{(1+m)\alpha-1} (z^\alpha+A_h)^{-1}A_h^{-m} P_hu^0\,\d z.
\end{align}
We shall establish an error estimate for $  u_h^r -  u^r$ by using the following technical lemma.
\begin{lemma}\label{lem:keylem}
The following estimate holds for $v\in H_0^1(\Omega)$ and $z\in \Sigma_{\theta}$ with $\theta\in(\frac\pi2,\pi)$:
\begin{equation}%\label{eqn:key}
    |z^\alpha| \|v\|_{L^2(\Omega)}^2 + \| \nabla v \|_{L^2(\Omega)}^2
    \le c\big|z^\alpha\|v\|_{L^2(\Omega)}^2 + (\nabla v,\nabla v)\big|.
\end{equation}
\end{lemma}
\begin{proof}
By \cite[Lemma 7.1]{FujitaSuzuki:1991}, we have that for any $z\in \Sigma_{\theta}$
\begin{equation*}%\label{eqn:key}
    |z| \|v\|_{L^2(\Omega)}^2 + \| \nabla v\|_{L^2(\Omega)}^2 \le c\big|z\|v\|_{L^2(\Omega)}^2 + (\nabla v,\nabla v)\big|.
\end{equation*}
%Alternatively, it follows from the inequality
%\begin{equation*}
%  \gamma |z| + \beta \leq \frac{|\gamma z+\beta|}{\sin\frac\theta2}\quad \mbox{for } \gamma,\beta\geq 0, z\in \Sigma_{\theta},
%\end{equation*}
%with the choice $\gamma=\|v\|_{L^2(\Omega)}^2$ and $\beta=\|\nabla v\|_{L^2(\Omega)}^2=(\nabla v,\nabla v)$.
%Since $\alpha\in(0,1)$, $z^\alpha\in\Sigma_{\theta}$ for all $z\in \Sigma_{\theta}$, and
%this completes the proof.
Alternatively, let $\gamma=\|v\|_{L^2(\Omega)}^2$ and $\beta=\|\nabla v\|_{L^2(\Omega)}^2
=(\nabla v,\nabla v)$ and $\arg(z) = \fy$, we have
\begin{equation*}
  |  z \gamma + \beta  |^2  \ge   (|z| \gamma \cos\fy + \beta)^2 +  (|z| \gamma \sin \fy)^2.
\end{equation*}
Therefore, we derive
\begin{equation*}
  |  z \gamma + \beta  |   \ge   |z| \gamma \sin \fy
\quad\mbox{and}\quad
  |  z \gamma + \beta  | ^2   \ge   (\beta\cos\fy + |z| \gamma)^2 + \beta^2\sin^2\fy \ge \beta^2\sin^2\fy.
\end{equation*}
Then for $\fy\in[\pi-\theta, \theta]$, we have
\begin{equation*}
2  |  z \gamma + \beta  |   \ge  ( |z| \gamma + \beta) \sin \fy \ge  ( |z| \gamma + \beta) \sin \theta.
\end{equation*}
Meanwhile, for $\fy\in[0,\pi-\theta]$, we have $\cos\fy \ge \cos(\pi-\theta) > 0$.
\begin{equation*}
 |  z \gamma + \beta  |  \ge   |z| \gamma \cos\fy + \beta \ge   |z| \gamma \cos(\pi-\theta) + \beta \ge  ( |z| \gamma + \beta)\cos(\pi-\theta).
 \end{equation*}
This completes the proof of the lemma.
\end{proof}

Let $w = (z^{\alpha} + A)^{-1} A^{-m} v$.
Appealing again to Lemma \ref{lem:keylem}, we obtain
\begin{equation*}
    |z^\alpha| \| A^m w \|_{L^2\II}^2 + \|\nabla A^m w\|_{L^2\II}^2
    \le c|((z^\alpha+A)A^m w,A^m w)|\le c\|   v \|_{L^2\II}\|A^m  w\|_{L^2\II}.
\end{equation*}
Consequently
\begin{equation}\label{eqn:wbound2}
    \|A^m w \|_{L^2\II} \le c|z^\alpha|^{-1}\| v \|_{L^2\II}\quad\mbox{and}\quad
    \|\nabla A^m w \|_{L^2\II} \le c|z^\alpha|^{-\frac12}\| v \|_{L^2\II}.
\end{equation}
In view of \eqref{eqn:wbound2} and the resolvent estimate \eqref{eqn:resol}, we can bound $\| w \|_{\dH 2}$ by
\begin{equation}\label{eqn:wbound3}
\begin{split}
   \|  A^m w \|_{\dot H^{2}\II} &=  \| A^{m+1} w \|_{L^2\II}
   =  \| (-z^\alpha +z^\alpha +A)(z^\alpha+A)^{-1}v \|_{L^2\II}\\
   &\le c(\| v \|_{L^2\II}+|z^\alpha|\|A^m w\|_{L^2\II})\le c\| v\|_{L^2\II}.
\end{split}
\end{equation}

%$$ \| w \|_{\dH2} \le c \|(z^{\alpha} + A)^{-1} v\|_{L^2\II} \le c   |z^\alpha|^{-1} \|v\|_{L^2\II}. $$

Next, we aim to develop an error estimate between $(z^\alpha +A)^{-1}A^{-m}u^0$
and its discrete analogue $(z^\alpha+A_h)^{-1}A_h^{-m}P_h u^0$
for $u^0\in L^2\II$. We begin with the following technical lemma.

\begin{lemma}\label{lem:wsigma}
Let $ u^0\in L^2(\Omega) $ and we define $\{ p_j \}_{j=1}^m$ such that
$$ p_1 =  A^{-1}u^0\quad \text{and}\quad p_{j} = A^{-1}p_{j-1}\quad \text{with}~~ j =1,2,\ldots,m.  $$
Moreover, we define $\{ p_{j,h} \}_{j=1}^m \subset X_h$  such that
$$ p_{1,h} =  A_h^{-1} P_h u^0\quad \text{and}\quad p_{j,h} =A_h^{-1} p_{j-1,h}\quad \text{with}~~ j =1,2,\ldots,m.  $$
Then there hold error estimates for $j=1,2\ldots,m$
\begin{equation}\label{eqn:wsigma}
    \|   p_{j,h} - p_{j} \|_{L^2(\Omega)} + h\| \nabla (p_{j,h}-p_{j})\|_{L^2(\Omega)}
    \le ch^{2j} \| u^0  \|_{L^2(\Omega)}
\end{equation}
and
\begin{equation}\label{eqn:wsigma-2}
    \|   p_{j,h} - p_{j} \|_{H^{-s}\II}  \le ch^{2j+s} \| u^0  \|_{L^2(\Omega)}\quad \text{with} ~~ 1 \le s \le 2m-2j+2.
\end{equation}
\end{lemma}
\begin{proof}
Let $\sigma_j = p_{j} - p_{j,h}$.
By the definition, we have $p_{1,h} = R_h p_1$ and hence derive
\begin{align*}
  \|\sigma_1 \|_{L^2(\Omega)}+h\|\nabla \sigma_1\|_{L^2(\Omega)}& \leq ch^2\| u^0 \|_{L^2(\Omega)}.
  %\|R_h\psi-\psi\|_{L^2(\Omega)}+h\|\nabla(R_h\psi-\psi)\|_{L^2(\Omega)}& \leq ch^q\|\psi\|_{H^q(\Omega)}\quad \forall\psi\in \dH q, q=1,2.\label{eqn:err-Rh}
\end{align*}
This and the negative norm estimate in Lemma \ref{eqn:err-ellip-1} lead to
\begin{align*}
  \|\sigma_1\|_{H^{-r}(\Omega)}  \le ch^{r+1}\|\sigma_1\|_{H^{1}(\Omega)} \le ch^{r+2} \|  u^0 \|_{L^2(\Omega)} ~~\text{with}~~  1\le r \le 2m.
\end{align*}
Next, we prove \eqref{eqn:wsigma} and \eqref{eqn:wsigma-2} by mathematical induction.
Assume that \eqref{eqn:wsigma} and \eqref{eqn:wsigma-2} holds for $j=k$. Then
Moreover, $p_j$ and $p_{j,h}$ respectively satisfy
\begin{equation*}
  \begin{aligned}
     (\nabla p_{k+1},\nabla \fy)&=(p_k,\fy),\quad \forall \fy \in H_0^1(\Omega),\\
     (\nabla p_{k+1,h},\nabla \fy_h)&=(p_{k,h},\fy_h),\quad \forall \fy_h\in X_h.
  \end{aligned}
\end{equation*}
Letting $\fy=\fy_h$ and subtracting these two identities yield the following error equation
\begin{equation}\label{eqn:sig-eq}
    (\nabla \sigma_{k+1}, \nabla \fy_h) = (\sigma_k,\fy_h), \quad \forall \fy_h\in X_h.
\end{equation}
Therefore, we have the following estimate
\begin{equation*}
\begin{split}
    \| \nabla \sigma_{k+1} \|_{L^2\II}^2
    &= (\nabla \sigma_{k+1}, \nabla (p_{k+1} - R_h p_{k+1})) + (\sigma_k, R_h p_{k+1} - p_{k+1}) + (\sigma_k, \sigma_{k+1})\\
    &\le  \|\nabla \sigma_{k+1}\|_{L^2\II} \|\nabla (p_{k+1} - R_h p_{k+1})\|_{L^2\II} \\
   &\quad + \|\nabla\sigma_k\|_{L^2\II} \| p_{k+1} - R_h p_{k+1}\|_{H^{-1}\II} + \|\sigma_k\|_{H^{-1}\II} \|\nabla\sigma_{k+1}\|_{L^2\II}. \\
\end{split}
\end{equation*}
According to \eqref{eqn:wsigma}, \eqref{eqn:wsigma-2} with $j=k$ and $s=-1$, \eqref{Error-solution-f} and Lemma \ref{eqn:err-ellip-1}, we derive
\begin{equation}\label{eqn:dsigma}
\begin{split}
   &\quad \| \nabla \sigma_{k+1} \|_{L^2\II}^2\\
    &\le c  \|\nabla (I-R_h)p_{k+1}\|_{L^2\II}^2
    + \|\nabla\sigma_k\|_{L^2\II} \| (I-R_h)p_{k+1}\|_{H^{-1}\II} + c\|\sigma_k\|_{H^{-1}\II}^2 \\
    &\le c h^{4k+2} \| u^0 \|_{L^2\II}^2
\end{split}
\end{equation}
Next, we show the error estimate in $H^{-r}\II$ with $ 0 \le r \le 2m-2k$ by duality argument.
For any $\fy \in \dH r$ we let $\phi = A^{-1} \fy$. Then we observe
\begin{equation*}
\begin{split}
    \|  \sigma_{k+1} \|_{H^{-r}\II}
   & = \sup_{\fy\in \dH r}\frac{\langle \sigma_{k+1}, \fy\rangle}{\| \fy \|_{\dH r}}
    = \sup_{\fy\in \dH r}\frac{( \nabla \sigma_{k+1}, \nabla \phi)}{\| \fy \|_{\dH r}} \\
   %& = \sup_{\fy\in \dH r}\frac{( \nabla \sigma_{k+1}, \nabla (\phi - R_h \phi)) + ( \sigma_{k},   R_h \phi)}{\| \fy \|_{\dH r}}  \\
   & = \sup_{\fy\in \dH r}\frac{( \nabla \sigma_{k+1}, \nabla (\phi - R_h \phi)) + ( \sigma_{k},   R_h \phi - \phi) + ( \sigma_{k},   \phi)}{\| \fy \|_{\dH r}}\\
%   & = \sup_{\fy\in \dH r}\frac{ \|\nabla \sigma_{k+1}\|_{L^2\II} \|\nabla (I - R_h) \phi \|_{L^2\II}
%   + \|\sigma_{k}\|_{\dH{1}}  \|(I - R_h)\phi\|_{\dH{-1}} + \|\sigma_{k}\|_{\dH{-2-r}}   \| \phi \|_{\dH {2+r}} }{\| \fy \|_{\dH r}}\\
%   & \le c h^{2k+2+r} \|  v \|_{L^2\II}
\end{split}
\end{equation*}
Using \eqref{eqn:dsigma} and \eqref{Error-solution-f}, we derive
\begin{equation*}
\begin{split}
   \sup_{\fy\in \dH r}\frac{( \nabla \sigma_{k+1}, \nabla (\phi - R_h \phi)) }{\| \fy \|_{\dH r}}
   &\le \sup_{\fy\in \dH r}\frac{\|\nabla \sigma_{k+1}\|_{L^2\II} \|\nabla (I - R_h) \phi \|_{L^2\II}}{\| \fy \|_{\dH r}} \\
   &\le \sup_{\fy\in \dH r}\frac{ ch^{2k+1} \| u^0 \|_{L^2\II} c h^{r+1}\| \phi \|_{\dH{r+2}}}{\| \fy \|_{\dH r}}  \le c h^{2k+2+r}.
\end{split}
\end{equation*}
Meanwhile, by duality between $\dH{-1}$ and $\dH1$, we apply \eqref{Error-solution-f} and \eqref{eqn:wsigma} with $j=k$ and $s=1$ to derive
\begin{equation*}
\begin{split}
   \sup_{\fy\in \dH r}\frac{( \sigma_{k},   R_h \phi - \phi)}{\| \fy \|_{\dH r}}
   &\le \sup_{\fy\in \dH r}\frac{ \|\sigma_k\|_{\dH{-1}} \| (I - R_h) \phi \|_{\dH1}}{\| \fy \|_{\dH r}} \\
   &\le \sup_{\fy\in \dH r}\frac{ ch^{2k+1} \| u^0 \|_{L^2\II} c h^{r+1}\| \phi \|_{\dH{r+2}}}{\| \fy \|_{\dH r}}  \le c h^{2k+2+r}.
\end{split}
\end{equation*}
Similarly, by means of the duality between $\dH{-2-r}$ and $\dH{2+r}$, we apply again \eqref{eqn:wsigma} with $j=k$ and $s=2+r$ to derive
\begin{equation*}
\begin{split}
   \sup_{\fy\in \dH r}\frac{( \sigma_{k},   \phi)}{\| \fy \|_{\dH r}}
   &\le \sup_{\fy\in \dH r}\frac{ \|\sigma_k\|_{\dH{-r-2}} \|   \phi \|_{\dH{2+r}}}{\| \fy \|_{\dH r}} \\
   &\le \sup_{\fy\in \dH r}\frac{ ch^{2k+r+2} \| u^0 \|_{L^2\II}\| \fy \|_{\dH r}}{\| \fy \|_{\dH r}}  \le c h^{2k+2+r}.
\end{split}
\end{equation*}
This completes the proof of the lemma.
\end{proof}

\begin{lemma}\label{lem:wbound}
Let $ u^0\in L^2(\Omega) $,  $z\in \Sigma_{\theta}$,
 $w=(z^\alpha+A)^{-1}p$ with $p=A^{-m}u^0$, and $w_h=(z^\alpha+A_h)^{-1} p_h$ with $p_h = A_h^{-m} P_h u^0$.
 Then there holds
\begin{equation}\label{eqn:wboundHa}
    \|  w_h-w \|_{L^2(\Omega)} + h\| \nabla (w_h-w)\|_{L^2(\Omega)}
    \le ch^{2m+2} \| u^0 \|_{L^2(\Omega)}.
\end{equation}
\end{lemma}
\begin{proof}
Let $e=w-w_h$ and $\sigma = p - p_h$.
Then Lemma \ref{lem:wsigma} implies the estimate
\begin{align}\label{eqn:err-sigma}
  \|\sigma\|_{L^2(\Omega)}+h\|\nabla \sigma\|_{L^2(\Omega)}& \leq ch^{2m}\|u^0\|_{L^2(\Omega)}.
 \end{align}
and the negative norm error estimate
\begin{align}\label{eqn:err-sigma-1}
  \|\sigma\|_{H^{-r}(\Omega)}    \le  ch^{r+2m}\|u^0\|_{L^2(\Omega)},~~\text{with}~~ 1\le r\le 2.
 \end{align}
Moreover, $w$ and $w_h$ respectively satisfy
\begin{equation*}
  \begin{aligned}
    z^\alpha(w,\fy )+(\nabla w,\nabla \fy)&=(p,\fy),\quad \forall \fy \in H_0^1(\Omega),\\
    z^\alpha(w_h,\fy_h)+(\nabla w_h,\nabla \fy_h)&=(p_h,\fy_h),\quad \forall \fy_h\in X_h.
  \end{aligned}
\end{equation*}
Subtracting these two identities yields the following error equation of $e$
\begin{equation}\label{eqn:worthog}
    z^\alpha(e,\fy_h) + (\nabla e, \nabla \fy_h) = (\sigma,\fy_h), \quad \forall \fy_h\in X_h.
\end{equation}
This and Lemma \ref{lem:keylem} imply that
\begin{equation*}
    \begin{split}
        |z^\alpha| \| e\|_{L^2(\Omega)}^2  + \| \nabla e \|_{L^2(\Omega)}^2
        & \le c \big|z^\alpha\| e \|_{L^2(\Omega)}^2 + (\nabla e, \nabla e)\big| \\
        & = c \left|z^\alpha(e,w-R_h w) + (\nabla e, \nabla(w-R_h w)) - (\sigma, w_h - R_h w) \right|
     \end{split}
\end{equation*}
By  using the Cauchy-Schwartz inequality and the duality between $\dH 1$ and $\dH{-1}$, we arrive at
\begin{equation}\label{eqn:control2}
    \begin{aligned}
     |z^\alpha| \| e \|_{L^2\II}^2 + \| \nabla e \|_{L^2(\Omega)}^2
     &\le c \big(|z^\alpha| \|w-R_h w\|_{L^2\II}^2
          + \|\nabla(w-R_h w)\|_{L^2\II}^2 + \|\sigma\|_{H^{-1}\II}^2  \big).
     \end{aligned}
\end{equation}
According to \eqref{eqn:wbound2}, \eqref{eqn:wbound3} and \eqref{eqn:err-sigma-1}, we derive
\begin{equation}\label{eqn:control3}
 \begin{aligned}
  &\quad   |z^\alpha| \|e\|_{L^2\II}^2 + \|\nabla e\|_{L^2\II}^2\\
      &\le ch^{4m+2} \big(|z^\alpha| \|(z^\alpha+A)^{-1} u^0\|_{\dH1}^2
          + \|(z^\alpha+A)^{-1}u^0\|_{\dH2}^2 + \| u^0 \|_{L^2\II}^2  \big) \\
      &\le ch^{4m+2}\| u^0 \|_{L^2\II}^2.
 \end{aligned}
\end{equation}
This gives the desired bound on $\|\nabla e\|_{L^2\II}$. Next, we bound
$\|e\|_{L^2\II}$ using a duality argument. For any fixed $\fy\in L^2\II$, we set
$\psi=(z^\alpha+A)^{-1}\fy$. %\quad \text{and}\quad \psi_h=(z^\alpha+A_h)^{-1}P_h\fy.$$
Then the preceding argument implies
\begin{equation}\label{eqn:control4}
     |z^\alpha| \|\psi-R_h \psi\|_{L^2\II}^2 + \|\nabla (\psi-R_h \psi)\|_{L^2\II}^2\le ch^2 \| \fy \|_{L^2\II}^2.
\end{equation}
we have by duality
\begin{equation*}
\|e \|_{L^2\II} = \sup_{\fy \in L^2\II}\frac{|(e,\fy)|}{\|\fy\|_{L^2\II}}
=\sup_{\fy \in L^2\II}\frac{|z^\alpha(e,\psi)+(\nabla e,\nabla \psi)|}{\|\fy\|_{L^2\II}}.
\end{equation*}
Then the desired estimate follows from \eqref{eqn:err-sigma-1}, \eqref{eqn:worthog}, \eqref{eqn:control3} and \eqref{eqn:control4} by
\begin{equation*}
    \begin{split}
      |z^\alpha(e,\psi)+(\nabla e,\nabla \psi)|
        & = |z^\alpha(e,\psi-R_h \psi)+(\nabla e,\nabla (\psi-R_h \psi))+(\sigma,R_h \psi)| \\
        & \le |z^\alpha(e,\psi-R_h \psi)+(\nabla e,\nabla (\psi-R_h \psi))|+ |(\sigma,R_h \psi-\psi)| + |(\sigma,\psi)| \\
        & \le |z^\alpha|^{\frac12}\|e\|_{L^2\II} |z^\alpha|^{\frac12}\| \psi-R_h \psi \|_{L^2\II}
         + \|\nabla e\|_{L^2\II}\| \nabla(\psi-R_h \psi) \|_{L^2\II}\\
         &\quad + \| \sigma \|_{H^{-1}\II} \| \nabla(R_h \psi-\psi)  \|_{L^2\II} + \| \sigma \|_{H^{-2}\II} \| \psi \|_{\dH2} \\
        & \le ch^{2m+2} \| u^0 \|_{L^2\II} \| \psi \|_{\dH2} \le ch^{2m+2} \| v \|_{L^2\II} \| \fy \|_{L^2\II}.
    \end{split}
\end{equation*}
This completes proof of the lemma.
\end{proof}

Now we can state error estimates for the regular part.
\begin{theorem}\label{thm:regular}
Let $u^r$ and $u_h^r$ be the functions defined by \eqref{eqn:split-2} and \eqref{eqn:semi-reg},
respectively. Then for $t>0$, there holds:
\begin{equation*}
  \| (u^r-u_h^r)(t)\|_{L^2\II} + h\| \nabla (u^r-u_h^r)(t)\|_{L^2\II}
  \le ch^{2m+2} t^{-(1+m)\alpha} \|u^0\|_{L^2\II}.
\end{equation*}
\end{theorem}
\begin{proof}
 For $v\in L^2(\Omega)$, by the solution representations, the error $e_h(t)$ can be represented as
\begin{equation*}
|(u^r-u_h^r)(t)|=\Big|\frac1{2\pi\mathrm{i}}\int_{\Gamma_{\theta,\kappa}} e^{zt} z^{(1+m)\alpha-1} (w_h(z)-w(z)) \,\d z \Big|,
\end{equation*}
with $w(z)=(z^\alpha +A)^{-1}A^{-m}u^0$ and $w_h(z)=(z^\alpha+A_h)^{-1}A_h^{-m}P_hu^0$.
By Lemma \ref{lem:wbound}, and taking $\kappa=t^{-1}$
in the contour $\Gamma_{\theta,\kappa}$, we have
\begin{equation*}
\|  (u^r-u_h^r)(t) \|_{L^2\II}\le ch^{2m+2} \| u^0 \|_{L^2\II} \int_{\Gamma_{\theta,\kappa}} e^{\Re(z)t} |z|^{(1+m)\alpha-1} \,|\d z|
\le ch^{2m+2} t^{-(1+m)\alpha} \|u^0\|_{L^2\II}.
\end{equation*}
A similar argument also yields the $H^1\II$-estimate.
\end{proof}

\begin{remark}\label{rem:regular-err-semi-1}
A slightly modification leads to the error estimate for very weaker initial data
$u^0\in \dH s$ with some $s\in[-1,0]$.
%One example is the Dirac measure concentrated on an interface in two dimension, which belongs to
%$\dH{-\frac12-\epsilon}$ for arbitrarily small $\epsilon>0$. See e.g., \cite{JinLazarovPasiackZhou:2013naa} and Example \ref{example62} in Section \ref{sec:numerics}.
In particular let $u^r$ and $u_h^r$ be the functions defined by \eqref{eqn:split-2} and \eqref{eqn:semi-reg},
respectively. Then for $t>0$, there holds
\begin{equation}\label{eqn:rem-est-1}
  \| (u^r-u_h^r)(t)\|_{L^2\II} + h\| \nabla (u^r-u_h^r)(t)\|_{L^2\II}
  \le ch^{2m+2+s} t^{-(1+m)\alpha} \|u^0\|_{\dH s}.
  \end{equation}
\end{remark}

\begin{remark}\label{rem:regular-err-semi-2}
The argument could be further extended to rougher initial data, such as the Dirac delta function $u^0 = \delta_{x_*} $
in two dimension with a fixed $x_* \in \Omega$. Then we consider the splitting
\begin{equation}\label{split-1-2}
    \begin{split}
    u^r(t)-u_h^r(t) & = F^r(t) u^0- F_h^r(t) P_h u^0  \\
& =  (F^r(t) u^0- F^r(t) P_h u^0) + (F^r(t) P_h u^0 - F_h^r(t)P_h u^0).
\end{split}
\end{equation}
The first term could be bounded using the smoothing property \eqref{eqn:reg-sm-est} and the $L^\infty$-stability of the $L^2$ projection (see \cite{Crouzeix-Thomee-1987})
\begin{equation*}
    \begin{split}
 \| F^r(t) u^0- F^r(t) P_h u^0\|_{L^2\II}
 &= \sup_{\phi\in L^2\II} \frac{(F^r(t) u^0- F^r(t) P_h u^0, \phi)}{\| \phi \|_{L^2\II}}
 %= \sup_{\phi\in L^2\II} \frac{\langle u^0 ,  (I - P_h) F^r(t) \phi\rangle}{\| \phi \|_{L^2\II}} \\
 \\
 & \le \sup_{\phi\in L^2\II} \frac{\|  (I - P_h) F^r(t) \phi \|_{L^\infty\II}}{\| \phi \|_{L^2\II}} \le C \sup_{\phi\in L^2\II} \frac{\|  (I - I_h) F^r(t) \phi \|_{L^\infty\II}}{\| \phi \|_{L^2\II}} \\
 & \le c h^{2m+1}\sup_{\phi\in L^2\II} \frac{\| F^r(t) \phi \|_{\dH{2m+2}}}{\| \phi \|_{L^2\II}}\le c h^{2m+1} t^{-(1+m)\alpha},
\end{split}
\end{equation*}
where we have used the $L^\infty$ error estimate for the Lagrange interpolation  (see Lemma \ref{lemma:err-I_h-Lyinfty}) in the second to last inequality.

Meanwhile, the second term in \eqref{split-1-2} could be bounded using the estimate \eqref{eqn:rem-est-1}
and the inverse inequality, i.e.,
\begin{equation*}
    \begin{split}
 \| F^r(t) P_h u^0 - F_h^r(t)P_h u^0 \|_{L^2\II} &\le  ch^{2m+2} t^{-(1+m)\alpha} \|P_h u^0\|_{L^2\II} \\
 &= ch^{2m+2} t^{-(1+m)\alpha}  \sup_{\phi\in L^2\II} \frac{|P_h\phi(x_*)|}{\| \phi \|_{L^2\II}}
 \le c h^{2m+1} t^{-(1+m)\alpha}.
\end{split}
\end{equation*}
As a result, we have the following error estimate for Dirac delta initial condition in two dimension
 $$ \|(u^r-u_h^r)(t) \|_{L^2\II} \le c h^{2m+1} t^{-(1+m)\alpha}.$$
This convergence rate is consistent with our numerical experiments, cf. Table \ref{table:exa52_space_m0}--\ref{table:exa52_space_m2}.
\end{remark}

\section{Time discretization}\label{sec:time}
In the preceding section, we have proposed a splitting of the exact solution into a time-dependent regular part plus several time-independent singular parts. Next, we shall develop and analyze a time stepping scheme for approximating the regular part using convolution quadrature.

We shall focus on time-stepping schemes with a uniform temporal mesh. Specifically, let $t_n=n\tau$, $n=0,1,\dots,N,$ be a uniform partition of the time interval $[0,T]$ with a time stepsize $\tau=N^{-1}T$,  $N\in\mathbb{N}$, and recall that the generating function of BDF method of order $k$, $k=1,\ldots,6$, is defined by
\begin{equation}\label{eqn:gen-fun-bdf}
  \delta_\tau(\zeta ):=\frac{\delta(\zeta)}{\tau}\quad \text{with}~~
  \delta(\zeta) = \sum_{j=1}^k \frac {1}{j} (1-\zeta )^j.
\end{equation}
The BDF$k$ method
is known to be $A(\vartheta_k)$-stable with angle $\vartheta_k= 90^\circ$, $90^\circ$, $86.03^\circ$,
$73.35^\circ$, $51.84^\circ$, $17.84^\circ$ for $k = 1,2,3,4,5,6$, respectively \cite[pp. 251]{HairerWanner:1996}.

Then we apply the following convolution quadrature to approximate the semidiscrete solution \eqref{eqn:semi-reg}:
\begin{align}\label{eqn:fully-reg-1}
U_h^{n,r} =    \frac{(-1)^m}{2\pi\rm{i}} \int_{\Gamma_{\theta,\kappa}^\tau} e^{zt_n}  \delta_\tau(e^{-z\tau})^{(1+m)\alpha-1} (\delta_\tau(e^{-z\tau})^\alpha+A_h)^{-1}A_h^{-m} P_h u^0\,\d z.
\end{align}
where the the contour $\Gamma_{\theta,\kappa}^\tau $ is
$\Gamma_{\theta,\kappa}^\tau :=\{ z\in \Gamma_{\theta,\kappa}:|\Im(z)|\le \tfrac{\pi}{\tau} \}$
oriented with an increasing imaginary part. The evaluation of the contour integral in \eqref{eqn:fully-reg-1} is equivalent to solving the following time-stepping scheme for $U_h^{n,r}$:
\begin{equation}\label{eqn:fully-reg-2}
\begin{aligned}
%&\omega_0^{(\alpha)} U_h^{0,r}  + A_hU_h^{0,r} =
%(-1)^m  \tau^{-(1+m)\alpha}\omega_0^{(1+m)\alpha-1} A_h^{-m} P_hv\\
\tau^{-\alpha} \sum_{j=0}^n \omega_j^{(\alpha)} U_h^{n-j,r}  + A_h U_h^{n,r} =(-1)^m  \tau^{-(1+m)\alpha}\omega_n^{(1+m)\alpha-1} A_h^{-m} P_hu^0, \quad \text{for}~~ 0 \le n\le N.
\end{aligned}
\end{equation}
Here the quadrature weights $\big(\omega_j^{(\beta)}\big)_{j=0}^\infty$
are given by the coefficients in the following power series expansion
\begin{equation}\label{eqn:delta}
\delta_\tau(\zeta)^\beta=\frac{1}{\tau^\beta}\sum_{j=0}^\infty \omega_j^{(\beta)} \zeta ^j.
\end{equation}
with the generating function \eqref{eqn:gen-fun-bdf}.
%For the special case $k=1$,
%the corresponding coefficients $\omega_j^{(\beta)}$ are
%given explicitly by the following recursion
%\begin{equation*}
%  \omega_0^{(\beta)} = 1, \quad \omega_{j}^{(\beta)}=-\frac{\beta-j+1}{j}\omega_{j-1}^{(\beta)}, \quad j=1,2,\ldots.
%\end{equation*}
Generally, those weights  can be evaluated efficiently via recursion or discrete
Fourier transform \cite{Podlubny:1999,Sousa:2012}.

Note that the time stepping scheme \eqref{eqn:fully-reg-2} begins with $n=0$,
which is different from the usual time stepping schemes for evolution problems.
The idea is closely related to correct the initial steps of the regular time stepping scheme
\cite{LubichSloanThomee:1996,CuestaLubichPalencia:2006,JinLiZhou:2017sisc,WangZhou:2020sinum}.
See a brief explanation in \cite[Appendix A]{JinLiZhou:2017sisc}.

The next Lemma shows the equivalence between the convolution quadrature \eqref{eqn:fully-reg-1} and
the time stepping scheme \eqref{eqn:fully-reg-2}.

\begin{lemma}\label{lem:time-step}
The function $U_h^{n,r}$ given by the contour integral \eqref{eqn:fully-reg-1} is the solution of
the time stepping scheme \eqref{eqn:fully-reg-2} for all $0\le n\le N$.
\end{lemma}

\begin{proof}
We begin with the time stepping scheme \eqref{eqn:fully-reg-2}.
By multiplying both sides of the relation \eqref{eqn:fully-reg-2} by $\zeta ^n$, summing over $n$ from $0$ to $\infty$
and collecting terms, we obtain
\begin{equation*}
\begin{aligned}
\sum_{n=0}^\infty \zeta^n \Big(\tau^{-\alpha} \sum_{j=0}^n \omega_j^{(\alpha)} U_h^{n-j,r}\Big)  + A_h
\sum_{n=0}^\infty U_h^{n,r} \zeta^n
&=(-1)^m  \tau^{-1} A_h^{-m} P_hu^0  \Big(\tau^{1-(1+m)\alpha} \sum_{n=0}^\infty \omega_n^{(1+m)\alpha-1} \zeta^n \Big)\\
&=(-1)^m  \tau^{-1} A_h^{-m} P_hu^0 \delta_\tau(\zeta)^{(1+m)\alpha-1}.
\end{aligned}
\end{equation*}
For any sequence $(v^n)_{n=1}^\infty$, we denotes its generating function by
$\widetilde v(\xi) =\sum_{n=0}^\infty v^n \zeta^n$.
The the leading term in the above relation can be written as
\begin{equation*}
\begin{aligned}
\sum_{n=0}^\infty \zeta^n \Big(\tau^{-\alpha} \sum_{j=0}^n \omega_j^{(\alpha)} U_h^{n-j,r}\Big)
&= \tau^{-\alpha} \sum_{j=0}^\infty\omega_j^{(\alpha)}  \zeta^j
\Big( \sum_{n=j}^\infty   U_h^{n-j,r} \zeta^{n-j}  \Big)\\
&=   \delta_\tau(\zeta)^\alpha \widetilde U_h^{r}(\zeta).
\end{aligned}
\end{equation*}
Therefore we obtain
\begin{equation*}
\begin{aligned}
\widetilde U_h^{r}(\zeta) = (-1)^m  \tau^{-1}   \delta_\tau(\zeta)^{(1+m)\alpha-1}
 ( \delta_\tau(\zeta)^\alpha + A_h)^{-1} A_h^{-m} P_hu^0.
\end{aligned}
\end{equation*}
Since $\widetilde U_h^{r}(\xi) $ is analytic with respect
to $\zeta $ in the unit disk on the complex plane $\mathbb{C}$, thus Cauchy's integral formula and
the change of variables $\zeta =e^{-z\tau}$ lead to the following representation for arbitrary $\varrho\in(0,1)$
\begin{equation}\label{repr-Wn}
\begin{aligned}
    U_h^{n,r} &= \frac{1}{2 \pi\mathrm{i}}\int_{|\zeta |=\varrho}\zeta ^{-n-1} \widetilde U_h^{r}(\zeta)\, \d\zeta
    = \frac{\tau}{2\pi\mathrm{i}}\int_{\Gamma^\tau} e^{zt_{n}}\widetilde  U_h^{r}(e^{-z\tau})\, \d z\\
    &= \frac{(-1)^m}{2\pi\mathrm{i}}\int_{\Gamma^\tau}
    e^{zt_{n}}\delta_\tau(e^{-z\tau})^{(1+m)\alpha-1}
 ( \delta_\tau(e^{-z\tau})^\alpha + A_h)^{-1} A_h^{-m} P_hu^0\, \d z
    \end{aligned}
\end{equation}
where $\Gamma^\tau$ is given by
$\Gamma^\tau:=\{ z=-\frac{\log \varrho}{\tau}+\mathrm{i} y: \, y\in{\mathbb R}\,\,\,\mbox{and}\,\,\,|y|\le \frac{\pi}{\tau} \}$.

Note that $\delta_\tau(e^{-z\tau})^{(1+m)\alpha-1}
 ( \delta_\tau(e^{-z\tau})^\alpha + A_h)^{-1} $ is analytic for $z\in \Sigma_{\theta,\kappa}^\tau$, which is a region enclosed by $\Gamma^\tau$, $\Gamma^\tau_{\theta,\kappa}$ and the two lines $\Gamma_{\pm}^\tau:
={\mathbb R}\pm \mathrm{i}\frac{\pi}{\tau}$ (oriented from left to right).
Using the periodicity of $e^{-z\tau}$ and Cauchy's theorem, we deform the contour $\Gamma^\tau$
to $\Gamma_{\theta,\kappa}^\tau$ in the integral \eqref{repr-Wn} to obtain the desired representation \eqref{eqn:fully-reg-1}.
\end{proof}

Finally, we study the error of convolution approximation. To this end, we
need the following lemma on the sectorial property and approximation property
of the generating function $\delta_\tau(\zeta)$.
See the detailed proof in \cite[Lemma B.1]{JinLiZhou:2017sisc}.
\begin{lemma}\label{lem:cq-delta}
For any $\varepsilon$, there exists $\theta_\varepsilon\in (\frac\pi2,\pi)$ such that for any
$\theta\in (\frac\pi2,\theta_\varepsilon)$,
there exist positive constants $c,c_1,c_2$ $($independent of $\tau$$)$ such that
\begin{equation*}
  \begin{aligned}
& c_1|z|\leq
|\delta_\tau(e^{-z\tau})|\leq c_2|z|,
&&\delta_\tau(e^{-z\tau})\in \Sigma_{\pi-\vartheta_k+\varepsilon}, \\
& |\delta_\tau(e^{-z\tau})-z|\le c\tau^k|z|^{k+1},
&& |\delta_\tau(e^{-z\tau})^\alpha-z^\alpha|\leq c\tau^k|z|^{k+\alpha},
&& \forall\, z\in \Gamma_{\theta,\kappa}^\tau,
\end{aligned}
\end{equation*}
where $\sigma>0$ and the contour $\Gamma_{\theta,\kappa}^\tau \subset \mathbb{C}$ is defined by
$$\Gamma_{\theta,\kappa}^\tau :=
\{ z=\rho e^{\pm\mathrm{i}\theta}: \rho\ge \kappa, |\Im(z)|\le \tfrac{\pi}{\tau} \}\cup
\left\{z=\kappa e^{\rm i\varphi}: |\varphi |\le \theta\right\}.$$
\end{lemma}

\begin{theorem}\label{thm:fully-err}
Let $U_h^{n,r}$ be the function defined by the convolution quadrature \eqref{eqn:fully-reg-1},
and $u_h^r$ be the function defined by the contour integral \eqref{eqn:semi-reg}. Then we have
$$  \| U_h^{n,r} - u_h^r(t_n) \|_{L^2\II} \le c \tau^k t_n^{-k-m\alpha} \| u^0 \|_{\dH s}\quad
\text{for any}~~ s\in [-1,0].$$
\end{theorem}
\begin{proof}
Let $K(z)= z^{(1+m)\alpha-1} (z^\alpha + A_h)^{-1}$. Then we may split the error as
\begin{equation*}
  \begin{aligned}
  U_h^{n,r} - u_h^r(t_n) &= \frac{(-1)^m}{2\pi\mathrm{i}} \int_{\Gamma^\tau_{\theta,\kappa}}
  e^{zt_n} \Big(K(z) - K(\delta_\tau(e^{-z\tau}))\Big) A_h^{-m} P_hu^0  \, \d  z\\
  &\quad + \frac{(-1)^m}{2\pi\mathrm{i}} \int_{\Gamma_{\theta,\kappa}\backslash\Gamma^\tau_{\theta,\kappa}}
  e^{zt_n} K(z) A_h^{-m} P_hu^0 \, \d z =: I_1 + I_2.
\end{aligned}
\end{equation*}
Using the resolvent estimate \eqref{eqn:resol} and Lemma \ref{lem:cq-delta}, we derive
$$ \| K(z) - K(\delta_\tau(e^{-z\tau}) \|_{L^2\II \rightarrow L^2\II} \le c \tau^k |z|^{m\alpha+k-1} .$$
As a result,
we choose $\kappa=t_n^{-1}$ in the contour $\Gamma_{\theta,\kappa}^\tau$, we obtain an estimate for $I_1$:
\begin{equation*}
  \begin{aligned}
 \|  I_1 \|_{L^2\II}&\le
 c \tau^k  \| A_h^{-m} P_h u^0 \|_{L^2(\Omega)} \Big(\int_{\kappa}^\infty  e^{\kappa t_n\cos\theta} \rho^{m\alpha+k-1}\,\d \rho
     + \int_{-\theta}^{\theta}   e^{\kappa t_n \cos\varphi} \kappa^{m\alpha+k}\,\d\varphi \Big)\\
     &\le c \tau^k(t_n^{-m\alpha-k}+\kappa^{m\alpha+k}) \| A_h^{-m} P_h u^0\|_{L^2(\Omega)}
     \le c \tau^k t_n^{-k-m\alpha}  \|A_h^{-m} P_h u^0 \|_{L^2(\Omega)}\\
     & \le c \tau^k t_n^{-k-m\alpha} \| u^0 \|_{\dH s},
\end{aligned}
\end{equation*}
for any $s\in[-1,0]$. The last inequality follows from the stability of $P_h$ in $\dH s$ for $s\in[-1,0]$:
$$  \|A_h^{-m} P_h u^0  \|_{L^2(\Omega)} \le  c \|A_h^{-1} P_h u^0  \|_{\dH1} \le c \| P_h u^0 \|_{\dH{-1}} \le c \|  u^0  \|_{\dH{-1}}.  $$

Meanwhile, for the term ${\rm I}_2$, we  apply
the resolvent estimate \eqref{eqn:resol} and Lemma \ref{lem:cq-delta} to derive
\begin{equation*}
\begin{aligned}
 \|  I_1 \|_{L^2\II} &\le  c \int_{\frac{\pi}{\tau\sin\theta}}^\infty e^{-c \rho t_n}  \rho^{m\alpha-1}
  \|A_h^{-m} P_h u^0   \|_{L^2(\Omega)} \, \d \rho\\
&\le c \tau^k  \int_{\frac{\pi}{\tau\sin\theta}}^\infty e^{-c \rho t_n}  \rho^{m\alpha+k-1}   \|A_h^{-m} P_h u^0  \|_{L^2(\Omega)} \, \d \rho\\
&\le c\tau^k t_n^{-m\alpha-k} \|A_h^{-m} P_h u^0  \|_{L^2(\Omega)}
\le  c \tau^k t_n^{-m\alpha-k}\|   u^0 \|_{\dH{-1}}.
\end{aligned}
\end{equation*}
This completes the proof of the theorem.
\end{proof}

In view of Remarks \ref{rem:regular-err-semi-1}--\ref{rem:regular-err-semi-2} and Theorem \ref{thm:fully-err}, we have the
following error estimate for the fully discrete solution.
\begin{corollary}\label{cor:fully-err}
Assume that $u^0 \in \dH s$ with some $s\in[-1,0]$, and $u$ is the solution to \eqref{eqn:fde} with $f=0$.
Let $U_h^{n,r}$ be the function defined by the convolution quadrature \eqref{eqn:fully-reg-1}.
%and $u^r$ be the functions defined by \eqref{eqn:split-2}.
Suppose that $\phi_{j,h} \in X_h$ is an approximation to $A^{-j} v$.
Then the fully discrete solution
\begin{equation}\label{eqn:fully}
 U_h^{n} = \sum_{j=1}^m (-1)^{j+1} \frac{t^{-j\alpha}}{\Gamma(1-j\alpha)} \phi_{j,h} + U_h^{n,r}
\end{equation}
satisfies the following error estimate
$$  \| U_h^{n} - u(t_n) \|_{L^2\II}
\le c \Big(h^{2m+2+s} t_n^{-(1+m)\alpha} + \tau^k t_n^{-k-m\alpha}\Big) \| v \|_{\dH s} + c\sum_{j=1}^m t_n^{-j\alpha} \| \phi_{j,h} -  A^{-j} v \|_{L^2\II} .$$
\end{corollary}

\section{Discussion on the approximation to the singular part $\sum_{j=1}^m u^s_j(t)$}\label{sec:singular}
The singular part $\sum_{j=1}^m u^s_j(t)$ of the solution in \eqref{eqn:split-2}, with
\begin{align}\label{singular-part}
u_j^s(t) = (-1)^{j+1} \frac{t^{-j\alpha}}{\Gamma(1-j\alpha)} A^{-j} u^0 ,
\end{align}
should be approximated separately. Since $u_j^s(t)$ can be computed by solving several elliptic equations, its computational cost is much smaller than the computation of the regular part (which requires solving an evolution problem; see the full discretization in the next section). Therefore, in general, the singular part can be solved by a much smaller mesh size without significantly increasing the overall computational cost.

In the following, we discuss several cases in which the singular part can be solved with high-order accuracy without using smaller meshes.

\begin{example}[Piecewise smooth initial data] \label{Example-1}
\upshape

If the initial value $u^0\in L^2(\Omega)$ is globally discontinuous (therefore nonsmooth) but piecewise smooth, separated by a smooth closed interface $\Gamma\subset\Omega$, then one can approximate $q_j=A^{-j}u^0$ by $q_{j,h}=A_h^{-j}P_hu^0$ using isoparametric finite element method with triangulations fitting the interface. The computation of $q_{j,h}=A_h^{-j}P_hu^0$ is equivalent to solving the following several standard elliptic equations:
$$
A_h q_{k,h} = q_{k-1,h},\quad k=1,\dots,j, \quad\mbox{with}\quad q_{0,h}=P_hu^0 .
$$
By denoting $\Omega=\Omega_1\cup\Omega_2\cup\Gamma$, where $\Omega_1$ and $\Omega_2$ are separated by a smooth interface $\Gamma$, the following high-order convergence can be achieved for isoparametric finite elements of degree $2m+1$ fitting the interface $\Gamma$:
\begin{align}\label{high-order-qj}
\| q_{j,h} - q_j \|_{L^2}
\le Ch^{2m+2}\|u^0\|_{H^{2m+2}_{\rm piecewise}(\Omega)} ,
\end{align}
where $\dot H^{2m+2}_{\rm piecewise}(\Omega)=\big\{ g \in L^2(\Omega): g|_{\Omega_j}\in \dot H^{2m+2}(\Omega_j)\,\,\mbox{for}\,\, j=1,2\big\}$.
This shows that the singular part in \eqref{singular-part} can be approximated with high-order accuracy for piecewise smooth initial data.

The error estimate in \eqref{high-order-qj} can be proved by using the following result (for isoparametric finite elements of degree $2m+1$ fitting the interface):
\begin{align}\label{Error-interface}
\| A_h^{-1}f - A^{-1} f \|_{L^2} \le Ch^{2m+2} \|f\|_{\dot H^{2m+2}_{\rm piecewise}(\Omega)},
\end{align}
which was originally proved in \cite{Li-Melenk-2010} for a bounded smooth domain $\Omega$ which contains the interface $\Gamma$. If $\Omega$ is a polygon which contains the interface $\Gamma$, then the interface is away from the corners of the polygon (therefore the functions in $\dot H^{2r+2}_{\rm piecewise}(\Omega)$ are locally in the classical Sobolev space $H^{2r+2}$ near the interface), it follows that the error estimates in \cite{Li-Melenk-2010} near the interface still hold for functions in $\dot H^{2r+2}_{\rm piecewise}(\Omega)$. Therefore, the combination of Proposition \ref{proposition:FEM} and the error estimates in \cite{Li-Melenk-2010} yields \eqref{Error-interface} for the triangulations satisfying \eqref{mesh-condition}.
Since
$$
\|q_{k}\|_{\dot H^{2m+2}_{\rm piecewise}(\Omega)}
=\|A^{-1}q_{k-1}\|_{\dot H^{2m+2}_{\rm piecewise}(\Omega)}
\le C\|q_{k-1}\|_{\dot H^{2m}_{\rm piecewise}(\Omega)} ,
$$
iterating this inequality yields that $\|q_{k}\|_{\dot H^{2m+2}_{\rm piecewise}(\Omega)}\le C\|u^0\|_{H^{2m+2}_{\rm piecewise}(\Omega)}$. By using this regularity and \eqref{Error-interface}, we have
\begin{align*}
\| q_{k,h} - q_k \|_{L^2}
&= \| A_h^{-1} q_{k-1,h} - A^{-1} q_{k-1} \|_{L^2} \\
&\le \| A_h^{-1} (q_{k-1,h} -P_hq_{k-1}) \|_{L^2} + \| A_h^{-1}P_hq_{k-1}  - A^{-1} q_{k-1}  \|_{L^2} \\
&\le \| q_{k-1,h} - P_hq_{k-1} \|_{L^2} + \| A_h^{-1}P_hq_{k-1}  - A^{-1} q_{k-1}  \|_{L^2} \\
&\le \| q_{k-1,h} - q_{k-1} \|_{L^2} + \| q_{k-1} - P_hq_{k-1} \|_{L^2} + \| A_h^{-1}P_hq_{k-1}  - A^{-1} q_{k-1}  \|_{L^2} \\
&\le \| q_{k-1,h} - q_{k-1} \|_{L^2} + Ch^{2m+2} \|q_{k-1}\|_{\dot H^{2m+2}_{\rm piecewise}(\Omega)} + Ch^{2m+2} \|q_{k-1}\|_{\dot H^{2m+2}_{\rm piecewise}(\Omega)} \\
&\le \| q_{k-1,h} - q_{k-1} \|_{L^2} + Ch^{2m+2} \|u^0\|_{H^{2m+2}_{\rm piecewise}(\Omega)}.
\end{align*}
By iterating this inequality for $k=1,\dots,j$, and using the following basic result:
$$ \| q_{0,h} - q_{0} \|_{L^2} = \| P_hu^0  -u^0\|_{L^2}\le Ch^{2m+2}\|u^0\|_{H^{2m+2}_{\rm piecewise}(\Omega)} , $$
we obtain the high-order convergence in \eqref{high-order-qj}.

\end{example}

\begin{example}[Dirac--Delta point source]  \label{Example-2}
\upshape

If the initial value is a Dirac--Delta point source centered at some interior point $x_0\in\Omega$, i.e., $u^0=\delta_{x_0}$, then the function
$$w_1 = A^{-1}u^0 - \hat q_1 , \quad\mbox{with}\quad \hat q_1(x)=\frac{1}{2\pi}\ln|x-x_0|, $$
is the solution of the following boundary value problem:
\begin{align}\label{PDE-w1}
\left\{\begin{aligned}
-\Delta w_1 &= 0 &&\mbox{in}\,\,\,\Omega,\\
w_1&= - \hat q_1    &&\mbox{for}\,\,\,x\in\partial\Omega,
\end{aligned}\right.
\end{align}
Let $\chi$ be a smooth cut-off function such that $\chi=1$ in a neighborhood of the boundary $\partial\Omega$ and $\chi=0$ in a neighborhood of $x_0$. Then $\chi \hat q_1\in C^\infty$ and
$$
\left\{\begin{aligned}
-\Delta(w_1-\chi \hat q_1) & = \Delta(\chi \hat q_1)\in C^\infty_0(\Omega) \subset \dot H^{2m}(\Omega) &&\mbox{in}\,\,\,\Omega,\\
w_1-\chi \hat q_1&= 0     &&\mbox{on}\,\,\,\partial\Omega.
\end{aligned}\right.
$$
This implies that $w_1-\chi \hat q_1 \in A^{-1}\dot H^{2m}(\Omega) = \dot H^{2m+2}(\Omega)$.
Since the explicit expression of $\Delta(\chi \hat q_1)$ is known, we can approximate $w_1-\chi \hat q_1$ by the finite element function $A_h^{-1}\Delta(\chi \hat q_1)$ and, correspondingly, approximate $q_1=A^{-1}u^0$ by
$q_{1,h}=\hat q_1 + \chi \hat q_1 + A_h^{-1}\Delta(\chi \hat q_1)$. The error of this approximation can be estimated as follows:
\begin{align*}
\|q_{1,h}-q_1\|_{L^2(\Omega)}
\le Ch^{2m+2}\|w_1-\chi\hat q_1\|_{\dot H^{2m+2}(\Omega)}
\le Ch^{2m+2} .
\end{align*}

Since $w_1$ is in $H^{2m+2}(\Omega)$, it follows that $A^{-1}w_1$ can be approximated by $A^{-1}_hw_1$ with an error bound of $O(h^{2m+2})$. Therefore, in order to compute $q_2=A^{-2}u^0=A^{-1}\hat q_1+A^{-1}w_1$ with an error bound of $O(h^{2m+2})$, it suffices to approximate $A^{-1}\hat q_1$ with the desired accuracy.
This can be done similarly as the approximation of $A^{-1}u^0$ by utilizing the following fact: The function
$$
\hat q_2(x) = -\frac{1}{2\pi}|x-x_0|^2\ln|x-x_0| + c|x-x_0|^2
$$
satisfies the equation $-\Delta \hat q_2 = \hat q_1$. Therefore, the function
$w_2 = A^{-1}\hat q_1 - \hat q_2$ is the solution of the following boundary value problem:
\begin{align}
\left\{\begin{aligned}
-\Delta w_2 &= 0 &&\mbox{in}\,\,\,\Omega,\\
w_2&= - \hat q_2    &&\mbox{for}\,\,\,x\in\partial\Omega ,
\end{aligned}\right.
\end{align}
which is in the same form as \eqref{PDE-w1}. Therefore, $A^{-1}\hat q_1$ can be computed with high-order accuracy similarly as the above-mentioned computation of $A^{-1}u^0$.
%{\color{blue}
Repeating this process will yield high-order approximations to $q_j=A^{-j}u^0$ with the following error bound:
\begin{align*}
\|q_{j,h}-q_j\|_{L^2(\Omega)} \le Ch^{2m+2} .
\end{align*}
%{\color{red}(We need to figure out an iterative proof for this repeated process?)}
%}

This shows that the singular part in \eqref{singular-part} can be approximated with high-order accuracy if the initial value is a Dirac--Delta point source.
\end{example}

\begin{example}[Dirac measure concentrated on an interface] \label{Example-3}
\upshape

If the initial value is a Dirac measure concentrated on an oriented interface $\Gamma\subset\Omega$, i.e., $u^0=\delta_{\Gamma}$ with
$$  \langle u^0 , v\rangle =  \langle \delta_{\Gamma} , v\rangle := \int_\Gamma v \, \d s \qquad \text{for all}~~ v \in \dH{\frac12+\mu},~\mu>0.$$
Then the function $A^{-1}u^0$ can be approximated by $A_h^{-1}P_hu^0$ with error of $O(h^{2m+2})$ in $L^2(\Omega)$ by using a locally refined mesh towards the interface $\Gamma$;  see e.g., \cite[Theorem 4.8]{LiWanYinZhao:2021}.
Similarly,
\begin{align*}
\|A^{-2}u^0 - A_h^{-1}P_hu^0\|_{L^2}
&\le
\|A^{-2}u^0 - A_h^{-1}P_hA^{-1}u^0\|_{L^2}
+ \|A_h^{-1}P_h(A^{-1}u^0-A_h^{-1}P_hu^0)\|_{L^2} \\
&\le
\|(A^{-1} - A_h^{-1}P_h)A^{-1}u^0\|_{L^2}
+ Ch^{2m+2} .
\end{align*}
Since $A^{-1}u^0$ is more regular than $u^0=\delta_\Gamma$, the locally refined mesh also yields optimal-order approximation
$$
\|(A^{-1} - A_h^{-1}P_h)A^{-1}u^0\|_{L^2} \le Ch^{2m+2}.
$$
The approximation to $A^{-j}u^0$ is similar. The details are omitted.

Therefore, the singular part in \eqref{singular-part} can be approximated with high-order accuracy if the initial value is a Dirac measure concentrated on an interface.
\end{example}

\begin{example}[Nonsmooth source term] \label{Example-4}
\upshape

We consider the subdiffusion problem with an inhomogeneous source term $g(t)f(x)$ where $g$ and $f$ are respectively temporally and spatially dependent functions, i.e.,
\begin{equation}\label{eqn:inhom}
\Dal u(t) + A u(t) = g(t)f \quad 0<t\le T,\quad \text{with}~~u(0)=0.
\end{equation}
By means of Laplace transform, the solution to \eqref{eqn:inhom} could be represented by
\begin{equation*}
 u(t) = \frac{1}{2\pi\mathrm{i}}\int_{\Gamma_{\theta,\kappa}} e^{zt}\hat g(z) (z^\alpha + A)^{-1} f \,\d z,
\end{equation*}
where $\hat g(z)$ denotes the Laplace transform of $g$.
Using the identity \eqref{identity}, we have the splitting
\begin{equation}\label{eqn:split-inhom}
  \begin{aligned}
    u(t) = \sum_{j=1}^m u_j^s(t) + u^r(t)
\end{aligned}
\end{equation}
where
\begin{equation*}
  \begin{aligned}
    &u_j^s(t) = \sum_{j=0}^{m-1} \frac{(-1)^{j}}{2\pi\mathrm{i}} A^{-(j+1)} f \int_{\Gamma_{\theta,\kappa}}
    e^{zt} \hat g(z) z^{j\alpha}\,\d z =:  \sum_{j=0}^{m-1} \Big((-1)^{j} A^{-(j+1)} f\Big) G_j (t)\\[5pt]
    &u^r(t) = \frac{(-1)^{m}}{2\pi\rm{i}} \int_{\Gamma_{\theta,\kappa}} e^{zt} z^{m\alpha} \hat g(z) (z^\alpha+A)^{-1}A^{-m} f \,\d z.
\end{aligned}
\end{equation*}
Next, we briefly introduce the approximation to $u^r(t)$. Using the argument in Section \ref{section:FEM},
we apply the semidiscrete finite element method:
\begin{equation*}
  \begin{aligned}
    u_h^r(t) = \frac{(-1)^{m}}{2\pi\rm{i}} \int_{\Gamma_{\theta,\kappa}} e^{zt} z^{m\alpha} \hat g(z) (z^\alpha+A_h)^{-1}A_h^{-m} P_h f \,\d z.
\end{aligned}
\end{equation*}
By assuming that $g \in C^{\lfloor m\alpha \rfloor + 1}[0,T]$, then the Taylor expansion and Lemmas \ref{lem:wsigma} and \ref{lem:wbound} imply the following error estimate
\begin{equation*}
  \begin{aligned}
    \|  u^r(t) - u_h^r(t) \|_{L^2\II} \le  c h^{2m+2} \|  f \|_{L^2\II} \Big(\sum_{\ell=0}^{\lfloor m\alpha \rfloor} |g^{(\ell)}(0)| t^{-m\alpha+\ell}
+\int_0^t |g^{\lfloor m\alpha \rfloor + 1}(t-s)|s^{-m\alpha+\lfloor m\alpha \rfloor}\,\d s\Big)
\end{aligned}
\end{equation*}
We then apply convolution quadrature to discretize in the time variable.
Let $\delta(\cdot)$ be the generating function of BDF$k$
method defined in \eqref{eqn:gen-fun-bdf}. By assuming that $g\in C^{K}[0,T]$ with $K=\lfloor (m-1)\alpha\rfloor + k$,
we apply the Taylor expansion to derive
\begin{equation*}
  \begin{aligned}
    u_h^r(t) & = \sum_{\ell=0}^{K} \frac{(-1)^{m}[ g^{(\ell)}(0)]}{2\pi\rm{i}} \int_{\Gamma_{\theta,\kappa}} e^{zt} z^{m\alpha - \ell} (z^\alpha+A_h)^{-1}  A_h^{-m} P_h f \,\d z \\
   &\qquad +  (-1)^{m}\frac{1}{2\pi\rm{i}} \int_{\Gamma_{\theta,\kappa}} e^{zt} z^{m\alpha} \widehat R_K(z) (z^\alpha+A_h)^{-1}  A_h^{-m} P_h f \,\d z
\end{aligned}
\end{equation*}
where $R_K(t) = \frac{t^K}{K!}*g^{(K+1)}$ denotes the remainder of the Taylor series.
Then we consider the time stepping approximation by convolution quadrature:
\begin{equation}\label{eqn:BDF-f}
  \begin{aligned}
    U_h^{r,n} & = \sum_{\ell=0}^{K} \frac{(-1)^{m}[ g^{(\ell)}(0)]}{2\pi\rm{i}} \int_{\Gamma_{\theta,\kappa}^\tau} e^{zt_n} \delta_\tau(e^{-z\tau})^{m\alpha - \ell} (\delta_\tau(e^{-z\tau})^\alpha+A_h)^{-1}  A_h^{-m} P_h f \,\d z \\
   &\qquad +  (-1)^{m} \frac{1}{2\pi\rm{i}} \int_{\Gamma_{\theta,\kappa}^\tau} e^{zt}
   \delta_\tau(e^{-z\tau})^{m\alpha} \widetilde R_K(\delta_\tau(e^{-z\tau})) (\delta_\tau(e^{-z\tau})^\alpha+A_h)^{-1}  A_h^{-m} P_h f \,\d z
\end{aligned}
\end{equation}
where $\widetilde R_K(\xi) =  \sum_{\ell=0}^\infty R_K(t_\ell) \xi^\ell$. Note that the fully discrete scheme could be solved via
a time stepping manner. Then using the argument in Section \ref{sec:time}, we have the error estimate
\begin{equation*}
  \begin{aligned}
   \| u_h^r(t_n)  -  U_h^{r,n} \|_{L^2\II} \le c \tau^k \Big( \sum_{\ell=0}^{K} |g^{(\ell)}(0)| t_n^{\ell-k-(m-1)\alpha}+
   \int_0^t |g^{(K+1)}(s)| (t-s)^{\ell-k-(m-1)\alpha} \,\d s\Big) \| f \|_{L^2\II}.
\end{aligned}
\end{equation*}

Similarly, we can approximate the function $G_j (t)$ in $u^s_j(t)$ by
using convolution quadrature generated by BDF$k$. Then we only need to solve an elliptic problem $A^{-(j+1)}f$ in $u^s_j(t)$ accurately,
see Example \ref{Example-1}-\ref{Example-3}.

Moreover, the above argument could be further generalized to the problem
\begin{equation}\label{eqn:inhom-m}
\Dal u(t) + A u(t) = \sum_{i=1}^{B} g_i(t) f_i \quad 0<t\le T,\quad \text{with}~~u(0)=0,
\end{equation}
where $g_i$ and $f_i$ are respectively temporally and spatially dependent functions for all $i=1,\ldots,B$.
\end{example}

%{\color{red}
%\begin{itemize}
%\item arbitarily high order using the similar argument
%\item Numerical Quadrature for contour integral
%\item accurate approximation to the singular part (several examples)
%\item time stepping
%\item $A=A(t)$
%\end{itemize}
%}

\section{Numerical experiments}\label{sec:numerics}
In the section, we present numerical experiments to support the theoretical analysis and to illustrate the high-order convergence of the proposed method for nonsmooth initial data.
Throughout, we consider a two-dimensional  subdiffusion model \eqref{eqn:fde} in a unit square domain $\Omega=(0,1)^2\subset \mathbb{R}^2$.
In our computation, the spatial mesh size be $h_j =  h_0/2^j$, and step size be $ \tau_j =  \tau^0/2^j$,
where $h_0$ and $\tau^0$ will be specified later. % in corresponding examples.
The errors are computed by the Cauchy difference
\begin{equation}
    \begin{aligned}
    E_{h_j}=  \|u_{\tau_{\rm ref}, h_j} - u_{\tau_{\rm ref}, h_{j+1}}\|_{L^2(\Omega)}, \qquad
    E_{\tau_j} =  \|u_{\tau_j, h_{\rm ref}} - u_{\tau_{j+1}, h_{\rm ref}}\|_{L^2(\Omega)},
    \end{aligned}
\end{equation}
and the convergence orders are computed by using the following formulae:
\begin{equation}
    \begin{aligned}
    \text{\rm spatial convergence order} &=  -({\log(E_{h_{j+1}}) - \log(E_{h_j})})/{\log 2}, \\
    \text{\rm temporal convergence order} &=  -({\log(E_{\tau_{j+1}}) - \log(\tau_{h_j})})/{\log 2}.\\
    \end{aligned}
\end{equation}
Let $r$ be the degree of finite elements in the spatial discretization, and $k$ be the order of the time-stepping method.
We illustrate the convergence of the time discretization for $k=1,2,3,4$ by fixing $m=1$ and $r=3$, and illustrate the convergence of the spatial discretization with different $m$ ($m=0,1$) and $r$ ($r=1,2,3$) by fixing $k=4$.
All the examples are performed by Firedrake \cite{Rathgeber2017Firedrake}, and the meshes are generated by Gmsh \cite{Geuzaine2009Gmsh}.\vskip10pt

\begin{example}[Dirac delta initial value]\label{example61}
\rm
In the first example, we test the very weak initial condition $u^0 = \delta_{x_0}$, where $\delta_{x_0}$ denotes the
Dirac delta measure concentrated at the single point $x_0 = (0.5+\epsilon, 0.5+\epsilon)$ with $\epsilon=10^{-4}$.
Here, a perturbation is given to move the source away from the vertex of the meshes.

To test the temporal convergence order of the fully discrete solution \eqref{eqn:fully} for different $k$,
we set $\tau_j = \tau^0/2^j$ with $\tau_0 = 1/32$ and a fixed spatial meshes $h_{\rm ref} = 1/512$.
The results of the $L^2$-errors are presented in  Figure \ref{fig:exa52_time}, and confirm $k$th-order convergence for the BDF$k$ method.
%\begin{table}[h!]
%\centering\scriptsize
%\caption{Example \ref{example61}: temporal error of \eqref{eqn:fully} with $m=1, r=3$.}\label{table:exa52_time}
%\vskip-5pt
%\begin{tabular}{ccccccccc}
%    \toprule
%     \multirow{2}{*}{$\alpha=0.6$}
%             &\multicolumn{2}{c}{$k=1$}
%             &\multicolumn{2}{c}{$k=2$}
%             &\multicolumn{2}{c}{$k=3$}
%             &\multicolumn{2}{c}{$k=4$}
%    \\\cmidrule(lr){2-3}\cmidrule(lr){4-5}\cmidrule(lr){6-7}\cmidrule(lr){8-9}
%         $\tau_j$ & $E_{\tau_j}$ &  conv. & $E_{\tau_j}$ &  conv. & $E_{\tau_j}$ &  conv. & $E_{\tau_j}$ &  conv. \\
%    \midrule
%       1/32 &  1.02e-06 &    -   &  6.03e-08 &    -   &  4.23e-09 &    -   &  3.79e-10 &    -  \\
%       1/64 &  5.09e-07 &   1.00 &  1.45e-08 &   2.06 &  4.98e-10 &   3.09 &  2.04e-11 &   4.21\\
%      1/128 &  2.55e-07 &   1.00 &  3.56e-09 &   2.03 &  6.04e-11 &   3.04 &  1.22e-12 &   4.06\\
%      1/256 &  1.28e-07 &   1.00 &  8.83e-10 &   2.01 &  7.44e-12 &   3.02 &  7.55e-14 &   4.02\\
%    \cmidrule{1-9}
%    \multicolumn{2}{c}{theor. conv.}
%                              &   1.00 &           &   2.00 &           &   3.00 &           &   4.00\\
%    \bottomrule
%\end{tabular}
%\end{table}

\begin{figure}[h!]
    \centering
    \includegraphics[width=.50\textwidth]{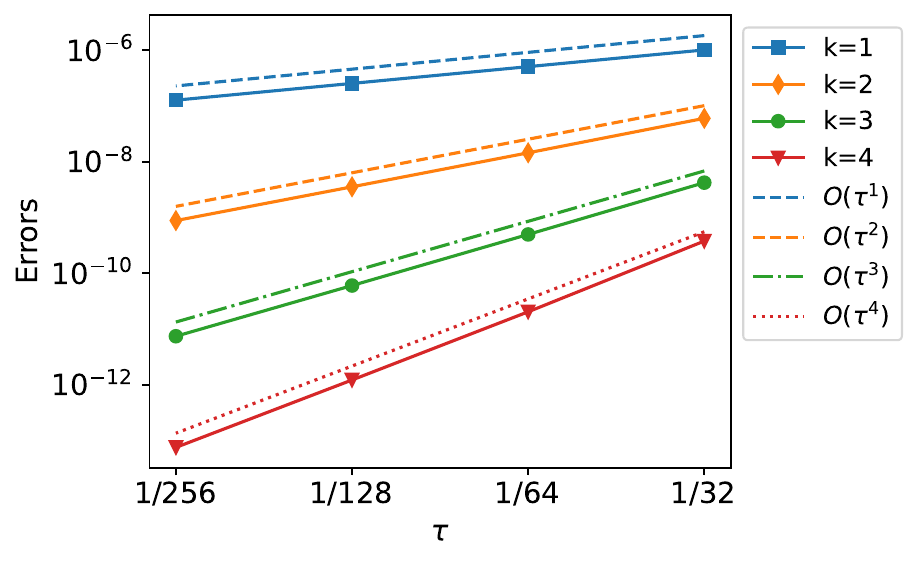}
    \caption{Example \ref{example61}: Errors of time discretization with $\alpha=0.6$, $m=1$, $r=3$.
    The dashed lines are $O(\tau^k)$.}
    \label{fig:exa52_time}
\end{figure}

To test the high-order convergence in space, we set $h_j=h_0/2^j$ with $h_0 = 1/16$ and $\tau_{\rm ref} = 1/1024$.
The meshes are refined by subdividing the triangles into four congruent sub-triangles (cf. Figure \ref{fig:exa52_space_meshes}).
In Table \ref{table:exa52_space_m0}, we test convergence the standard Galerkin finite element method, i.e. $m=0$,
using piecewise $r$-th order polynomials,
for both subdiffusion equation ($\alpha=0.6$) and normal diffusion equation $(\alpha=1)$.
The empirical convergence  for the fractional subdiffusion equation is always first-order, while that for the normal diffusion equation is of order $r+1$.
This interesting phenomenon is attributed to the infinite smoothing effect of the normal diffusion and the limited smoothing property of the fractional subdiffusion, cf. \eqref{eqn:reg-lim}.
Note that the Dirac delta function is in $\dH{-1-\mu}$ with any $\mu > 0$ and hence the solution to the subdiffusion equation \eqref{eqn:fde} belongs to $\dH{1-\mu}$.
In order to improve the convergence, we apply the splitting strategy \eqref{eqn:split-2} with $m=1$, approximate the regular part using the fully discrete scheme \eqref{eqn:fully-reg-2},
and compute the singular part using the method provided in Section \ref{Example-2}.
In Table \ref{table:exa52_space_m1}, we present the errors for both the regular part and the singular part.
The convergence for the regular part could be improved to $O(h^{\min(3,r+1)})$ and the singular part could be approximated with order $O(h^{\min(4, r+1)})$.
This convergence could be further improved by splitting one more singular term.
See Table \ref{table:exa52_space_m2} for the approximation with $m=2$.
These results fully supports our theoretical findings and the necessity of the proposed splitting method.

%we present the results of regular part and singular part for $\alpha=0.6$ when $m=1,2$. \ref{table:exa52_space_m2}
%It shows that the numerical convergence order of regular part is $\min\{2m+1, r+1\}$, and that of singular part is $\min\{2m+2, r+1\}$ as expected.
%Obviously, the new split method can achieve high order convergence.

\begin{figure}[h!]
    \centering
    \subfigure[$j=0$]{\includegraphics[width=.31\textwidth]{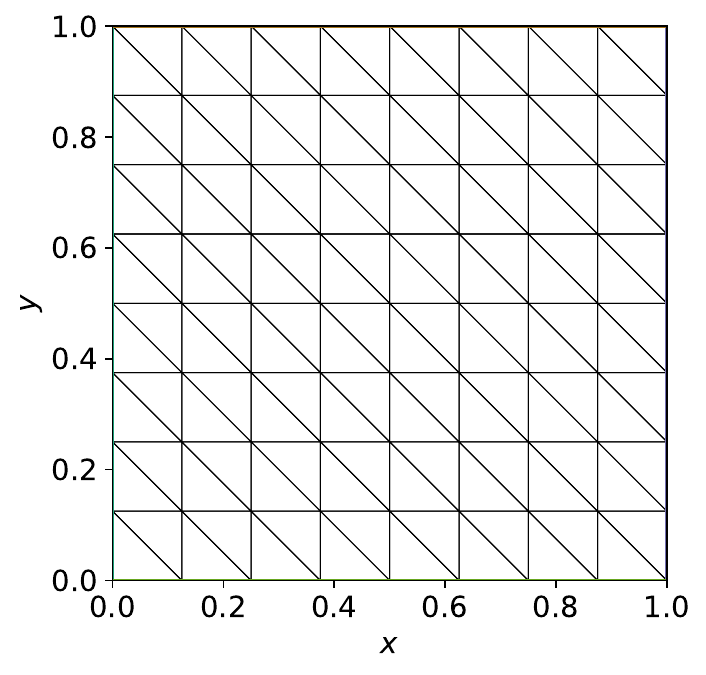}}
    \subfigure[$j=1$]{\includegraphics[width=.31\textwidth]{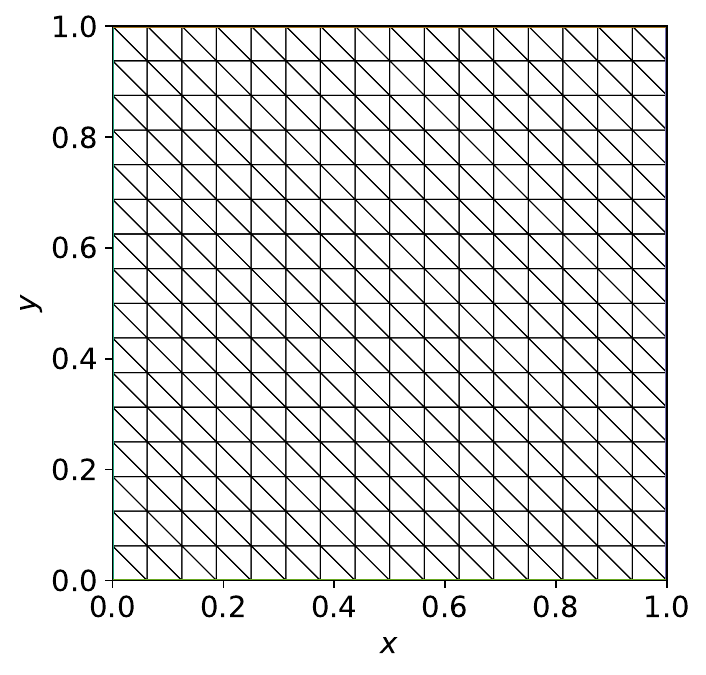}}
    \subfigure[$j=2$]{\includegraphics[width=.31\textwidth]{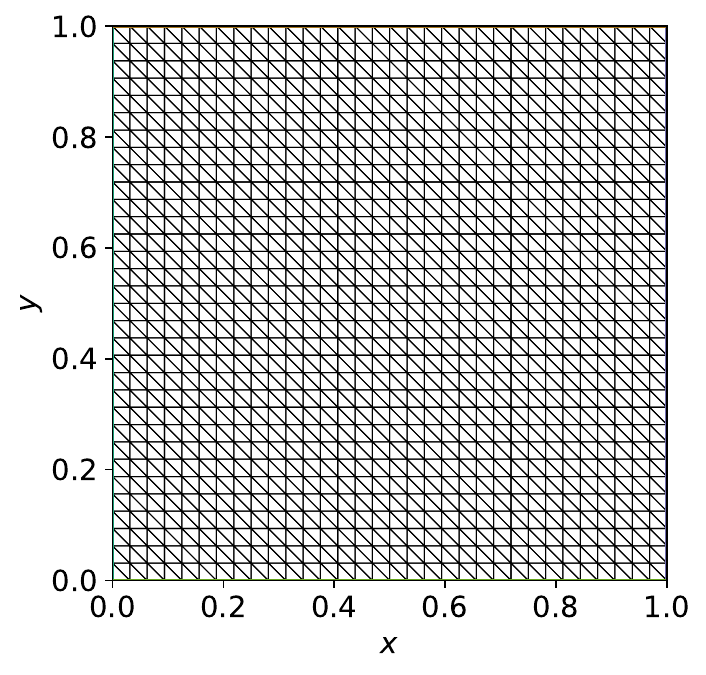}}
    % \subfigure[$j=2$]{\includegraphics[width=.30\textwidth]{exa52_mesh2_zoom2.pdf}}
    \caption{The meshes for Example \ref{example61}.}
    \label{fig:exa52_space_meshes}
\end{figure}

\begin{table}[h!]
\centering\scriptsize
\caption{Example \ref{example61}: comparison of $\alpha=0.6$ and $\alpha=1$, with $m=0$.}\label{table:exa52_space_m0}
\vskip-5pt
\begin{tabular}{ccccccccccccc}
    \toprule
             &\multicolumn{6}{c}{$\alpha=0.6$}
             &\multicolumn{6}{c}{$\alpha=1$}
    \\\cmidrule(lr){2-7}\cmidrule(lr){8-13}
             &\multicolumn{2}{c}{$r=1$}
             &\multicolumn{2}{c}{$r=2$}
             &\multicolumn{2}{c}{$r=3$}
             &\multicolumn{2}{c}{$r=1$}
             &\multicolumn{2}{c}{$r=2$}
             &\multicolumn{2}{c}{$r=3$}
    \\\cmidrule(lr){2-3}\cmidrule(lr){4-5}\cmidrule(lr){6-7}\cmidrule(lr){8-9}\cmidrule(lr){10-11}\cmidrule(lr){12-13}
        $h_j$ & $E_{h_j}$ &  conv. & $E_{h_j}$ &  conv. & $E_{h_j}$ &  conv. & $E_{h_j}$ &  conv. & $E_{h_j}$ &  conv. & $E_{h_j}$ &  conv.\\
    \midrule
     1/32 &  1.59e-04 &    -   &  1.02e-04 &    -   &  6.60e-05 &    -   &  1.79e-10 &    -   &  1.22e-13 &    -   &  1.40e-15 &    -  \\
     1/64 &  7.86e-05 &   1.02 &  4.92e-05 &   1.05 &  3.06e-05 &   1.11 &  4.58e-11 &   1.97 &  1.25e-14 &   3.28 &  8.40e-17 &   4.06\\
    1/128 &  3.86e-05 &   1.03 &  2.30e-05 &   1.10 &  1.36e-05 &   1.18 &  1.15e-11 &   1.99 &  1.46e-15 &   3.09 &  4.87e-18 &   4.11\\
    1/256 &  1.87e-05 &   1.04 &  1.04e-05 &   1.15 &  6.32e-06 &   1.10 &  2.89e-12 &   2.00 &  1.80e-16 &   3.03 &  2.80e-19 &   4.12\\
    \cmidrule{1-13}
    \multicolumn{2}{c}{theor. conv.}
                      &   1.00 &           &   1.00 &           &   1.00 &           &   2.00 &           &   3.00 &           &   4.00\\
    \bottomrule
\end{tabular}
\end{table}
\begin{table}[h!]
\centering\scriptsize
\caption{Example \ref{example61}: improved spatial convergence using \eqref{eqn:fully} with $m=1$.}\label{table:exa52_space_m1}
\vskip-5pt
\begin{tabular}{ccccccccccccc}
    \toprule
    \multirow{3}{*}{$\alpha=0.6$}
             &\multicolumn{6}{c}{regular part}
             &\multicolumn{6}{c}{singular part}
    \\\cmidrule(lr){2-7}\cmidrule(lr){8-13}
             &\multicolumn{2}{c}{$r=1$}
             &\multicolumn{2}{c}{$r=2$}
             &\multicolumn{2}{c}{$r=3$}
             &\multicolumn{2}{c}{$r=1$}
             &\multicolumn{2}{c}{$r=2$}
             &\multicolumn{2}{c}{$r=3$}
    \\\cmidrule(lr){2-3}\cmidrule(lr){4-5}\cmidrule(lr){6-7}\cmidrule(lr){8-9}\cmidrule(lr){10-11}\cmidrule(lr){12-13}
        $h_j$ & $E_{h_j}$ &  conv. & $E_{h_j}$ &  conv. & $E_{h_j}$ &  conv. & $E_{h_j}$ &  conv. & $E_{h_j}$ &  conv. & $E_{h_j}$ &  conv.\\
    \midrule
         1/32 &  3.06e-06 &    -   &  2.61e-08 &    -   &  2.53e-09 &    -   &  2.89e-04 &    -   &  3.92e-05 &    -   &  7.85e-06 &    -  \\
         1/64 &  7.71e-07 &   1.99 &  3.54e-09 &   2.88 &  3.16e-10 &   3.00 &  8.45e-05 &   1.77 &  5.65e-06 &   2.80 &  5.51e-07 &   3.83\\
        1/128 &  1.95e-07 &   1.98 &  4.76e-10 &   2.89 &  4.00e-11 &   2.98 &  2.24e-05 &   1.92 &  7.32e-07 &   2.95 &  3.39e-08 &   4.02\\
        1/256 &  4.97e-08 &   1.97 &  6.42e-11 &   2.89 &  5.33e-12 &   2.91 &  5.68e-06 &   1.98 &  9.25e-08 &   2.98 &  2.09e-09 &   4.02\\
    \cmidrule{1-13}
    \multicolumn{2}{c}{theor. conv.}
                          &   2.00 &           &   3.00 &           &   3.00 &           &   2.00 &           &   3.00 &           &   4.00\\
    \bottomrule
\end{tabular}
\end{table}
\begin{table}[h!]
\centering\scriptsize
\caption{Example \ref{example61}: improved spatial convergence using \eqref{eqn:fully} with $m=2$.} \label{table:exa52_space_m2}
\vskip-5pt
\begin{tabular}{ccccccccccccc}
    \toprule
    \multirow{3}{*}{$\alpha=0.6$}
             &\multicolumn{6}{c}{regular part}
             &\multicolumn{6}{c}{singular part}
    \\\cmidrule(lr){2-7}\cmidrule(lr){8-13}
             &\multicolumn{2}{c}{$r=1$}
             &\multicolumn{2}{c}{$r=2$}
             &\multicolumn{2}{c}{$r=3$}
             &\multicolumn{2}{c}{$r=1$}
             &\multicolumn{2}{c}{$r=2$}
             &\multicolumn{2}{c}{$r=3$}
    \\\cmidrule(lr){2-3}\cmidrule(lr){4-5}\cmidrule(lr){6-7}\cmidrule(lr){8-9}\cmidrule(lr){10-11}\cmidrule(lr){12-13}
        $h_j$ & $E_{h_j}$ &  conv. & $E_{h_j}$ &  conv. & $E_{h_j}$ &  conv. & $E_{h_j}$ &  conv. & $E_{h_j}$ &  conv. & $E_{h_j}$ &  conv.\\
    \midrule
         1/32 &  5.69e-07 &    -   &  2.08e-09 &    -   &  3.29e-11 &    -   &  2.77e-04 &    -   &  3.77e-05 &    -   &  7.49e-06 &    -  \\
         1/64 &  1.43e-07 &   1.99 &  2.61e-10 &   3.00 &  1.98e-12 &   4.05 &  8.11e-05 &   1.77 &  5.41e-06 &   2.80 &  5.27e-07 &   3.83\\
        1/128 &  3.60e-08 &   1.99 &  3.26e-11 &   3.00 &  1.14e-13 &   4.12 &  2.14e-05 &   1.92 &  7.01e-07 &   2.95 &  3.25e-08 &   4.02\\
        1/256 &  9.13e-09 &   1.98 &  4.08e-12 &   3.00 &  7.10e-15 &   4.00 &  5.44e-06 &   1.98 &  8.86e-08 &   2.98 &  2.00e-09 &   4.02\\
    \cmidrule{1-13}
    \multicolumn{2}{c}{theor. conv.}
                          &   2.00 &           &   3.00 &           &   4.00 &           &   2.00 &           &   3.00 &           &   4.00\\
    \bottomrule
\end{tabular}
\end{table}
%\begin{figure}[h!]
%    \centering
%    \subfigure[$\alpha=1.0, m=0$                ]{\includegraphics[width=.32\textwidth]{{exa52_space_cauchy_k=4_m=0_diff_alpha1}.pdf}}
%    \subfigure[$\alpha=0.6, m=1$ (regular part) ]{\includegraphics[width=.32\textwidth]{{exa52_space_cauchy_k=4_m=1_alpha0.6_reg}.pdf}}
%    \subfigure[$\alpha=0.6, m=2$ (regular part) ]{\includegraphics[width=.32\textwidth]{{exa52_space_cauchy_k=4_m=2_alpha0.6_reg}.pdf}}
%    \subfigure[$\alpha=0.6, m=0$                ]{\includegraphics[width=.32\textwidth]{{exa52_space_cauchy_k=4_m=0_diff_alpha0.6}.pdf}}
%    \subfigure[$\alpha=0.6, m=1$ (singular part)]{\includegraphics[width=.32\textwidth]{{exa52_space_cauchy_k=4_m=1_alpha0.6_sig}.pdf}}
%    \subfigure[$\alpha=0.6, m=2$ (singular part)]{\includegraphics[width=.32\textwidth]{{exa52_space_cauchy_k=4_m=2_alpha0.6_sig}.pdf}}
%    \caption{Example \ref{example61}: spatial convergence of  \eqref{eqn:fully} with different $m$ and $r$.}
%    \label{fig:exa52_space}
%\end{figure}

\FloatBarrier

\end{example}

\begin{example}[Dirac measure concentrated on an interface]\label{example62}

\rm
In the second example, we test the initial condition $u^0 = \delta_\Gamma$, where
$\delta_\Gamma$ denotes the Dirac measure concentrated on an oriented interface
$\Gamma = \overrightarrow{x_1x_2}$ with $x_1=(0.25, 0.75)$, $x_2=(0.75. 0.5)$, cf. Figure \ref{fig:exa53_domain} (a).
\begin{figure}[h!]
    \centering
    \subfigure[Line segment $\Gamma$]{\includegraphics[width=.32\textwidth]{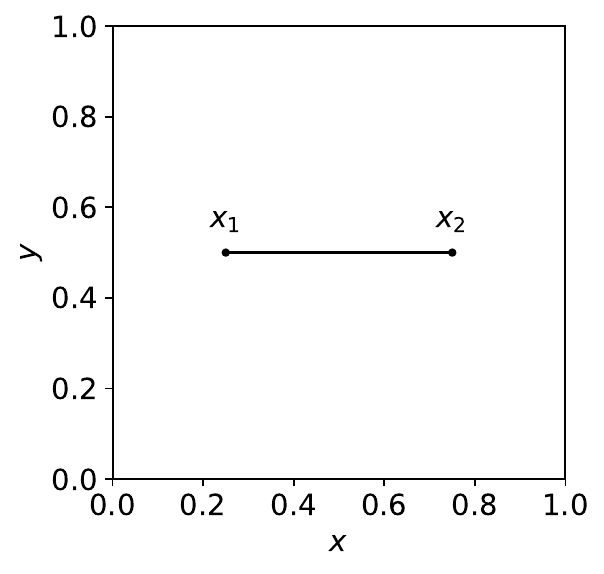}}
    \subfigure[$r=2,j=1$]{\includegraphics[width=.32\textwidth]{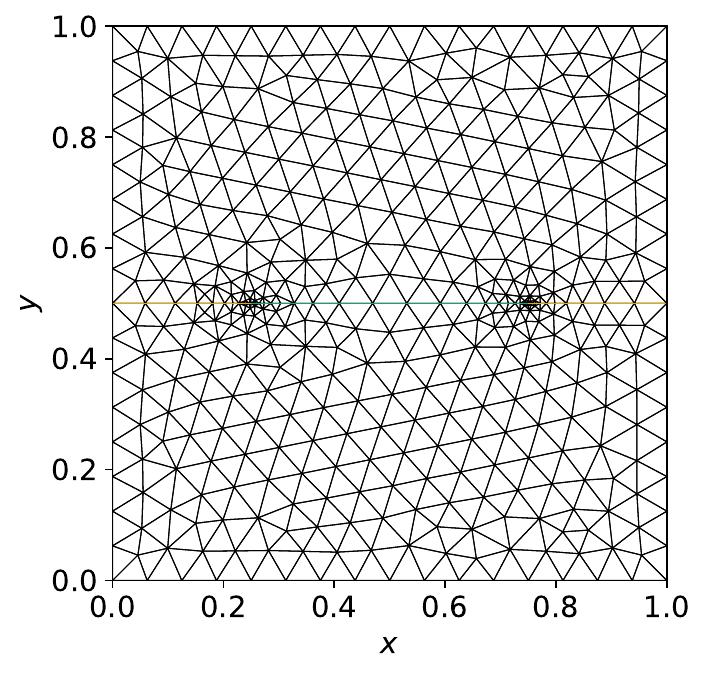}}
    \subfigure[$r=2,j=2$]{\includegraphics[width=.32\textwidth]{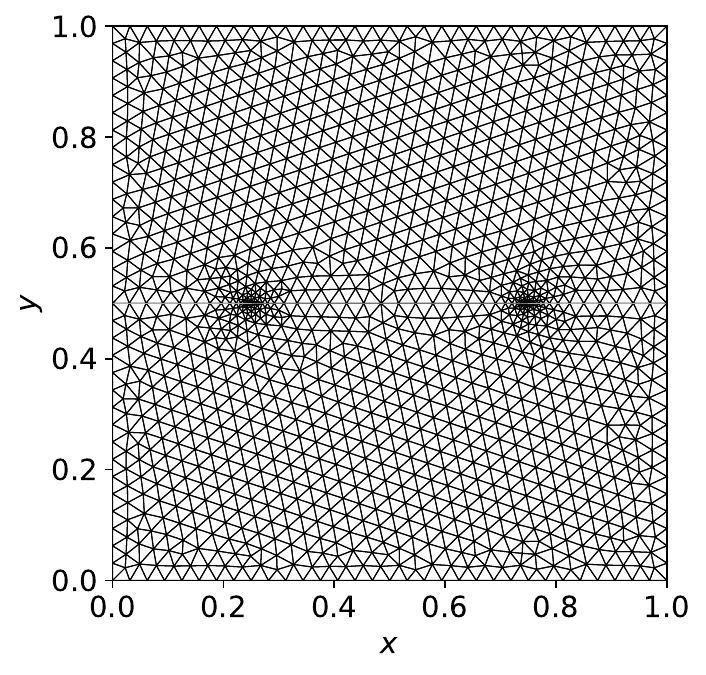}}
    \caption{Example \ref{example62}: Line segment $\Gamma$ and the graded mesh for the singular part.}\label{fig:exa53_domain}
\end{figure}

In order to reduce the computational cost,
we use quasi-uniform meshes and locally graded meshes for the time-dependent
regular part and the steady singular part, respectively.
To generate the quasi-uniform meshes, we generate the initial mesh with mesh size $h_0=1/8$ by Gmsh, and refine the mesh
several times to reach the mesh size $h_j=h_0/2^j$.
For the singular part, we generate the $j$th-level graded meshes with the local cell
diameter $\hbar(x)$ in sub-domain $B(x_i, d_0)$ as
\begin{equation}
    \hbar(x) \backsim \left\{
        \begin{aligned}
            &|x - x_i|^{1-\gamma} h_j,  &&\text{for}\quad h_\star \le |x-x_i| \le d_0, \\
            &h_*,  &&\text{for}\quad |x-x_i| \le h_* \backsim h_j^{1/\gamma}.\\
        \end{aligned}
    \right.
\end{equation}
where $\gamma \in (0, 1/r)$.
The graded mesh for approximating the singular part are presented in Figure \ref{fig:exa53_domain} (b) and (c).
Note that the refinement strategy used here is different from the method proposed in \cite{LiWanYinZhao:2021}, where the local mesh size for the $j$th-level graded meshes
in the neighborhood of $x_0$ and $x_1$ is
\begin{equation}
    \hbar(x) \backsim \left\{
        \begin{aligned}
            &|x - x_i|(1-c_p) h_j,  &&\text{for}\quad h_\star \le |x-x_i| \le d_0, \\
            &h_\star,  &&\text{for}\quad |x-x_i| \le h_\star \backsim \kappa_p^jh_j,\\
        \end{aligned}
    \right.
\end{equation}
where $c_p = 2^{-r/a}$ with  $a \in (0, 1)$ \cite[Algorithm 4.1]{LiWanYinZhao:2021}.
As proved in \cite[Theorem 3.8]{LiWanYinZhao:2021},
the solution of the Possion equation with line Dirac source belongs to weighted Sobolev space
\begin{equation}
    \mathcal{K}_{a+1}^{l+1}(B(x_i, d)\backslash\Gamma) := \{v: \rho^{|s| - a - 1} D^s v \in L^2(B(x_i, d)\backslash\Gamma), \forall |s| \le l + 1\}
\end{equation}
for any $l\ge 1$ and $a\in (0, 1)$ in the neighborhood of $x_1$ and $x_2$.
Though the refine methods given above are different,
both of them can resolve the singularity around the end point of $\Gamma$ and obtain optimal convergence order.

%\begin{figure}[h!]
%    \centering
%    \subfigure[$r=2,j=1$]{\includegraphics[width=.32\textwidth]{meshes53/exa53__r2_mesh1.pdf}}
%    \subfigure[$r=2,j=2$]{\includegraphics[width=.32\textwidth]{meshes53/exa53__r2_mesh2.pdf}}
%    \subfigure[$r=2,j=4$]{\includegraphics[width=.34\textwidth]{meshes53/exa53_r2_mesh4_zoomA.pdf}}
%    \caption{The meshes for solving the singular part.}
%    \label{fig:exa53_space_meshes}
%\end{figure}

To test the temporal convergence order, we let the step sizes be $\tau_j = \tau_0/2^j$ with $\tau_0 = 1/32$ and fixed the spatial mesh size $h_{\rm ref} = h_6 = 1/512$.
As shown in Figure \ref{fig:exa53_time},
the convergence order of BDF$k$ scheme is $O(\tau^k)$, which agrees well with our theoretical result in Corollary \ref{cor:fully-err}.

To test the convergence in space, we first compare the numerical results of $\alpha=0.6$ and $\alpha=1$ using the standard finite element method (i.e., $m=0$) with uniform meshes.
As shown in Table \ref{table:exa53_space_m0},
the convergence for the fractional diffusion is at most second-order even if we use high-order finite element methods,
while the convergence order of the normal diffusion is $r+1$.
In order to improve the convergence, we apply the splitting method with $m=1$.
The empirical errors of regular part and singular part are presented in Table \ref{table:exa53_space_m1}.
Our numerical experiments indicate the optimal convergence rate for the $P^r$ finite element method with $r=2,3$.
%\begin{table}[h!]
%\centering\scriptsize
%\caption{Example \ref{example62}: temporal error of  \eqref{eqn:fully} with $m=1, r=3$.}\label{table:exa53_time}
%\vskip-5pt
%\begin{tabular}{ccccccccc}
%    \toprule
%    \multirow{2}{*}{$\alpha=0.6$}
%             &\multicolumn{2}{c}{$k=1$}
%             &\multicolumn{2}{c}{$k=2$}
%             &\multicolumn{2}{c}{$k=3$}
%             &\multicolumn{2}{c}{$k=4$}
%    \\\cmidrule(lr){2-3}\cmidrule(lr){4-5}\cmidrule(lr){6-7}\cmidrule(lr){8-9}
%         $\tau_j$ & $E_{\tau_j}$ &  conv. & $E_{\tau_j}$ &  conv. & $E_{\tau_j}$ &  conv. & $E_{\tau_j}$ &  conv. \\
%    \midrule
%           1/32 &  4.56e-07 &    -   &  2.70e-08 &    -   &  1.89e-09 &    -   &  1.69e-10 &    -  \\
%           1/64 &  2.29e-07 &   1.00 &  6.51e-09 &   2.05 &  2.23e-10 &   3.08 &  9.10e-12 &   4.21\\
%          1/128 &  1.15e-07 &   1.00 &  1.60e-09 &   2.03 &  2.71e-11 &   3.04 &  5.51e-13 &   4.05\\
%          1/256 &  5.73e-08 &   1.00 &  3.96e-10 &   2.01 &  3.33e-12 &   3.02 &  3.07e-14 &   4.17\\
%    \cmidrule{1-9}
%    \multicolumn{2}{c}{theor. conv.}
%                            &   1.00 &           &   2.00 &           &   3.00 &           &   4.00\\
%    \bottomrule
%\end{tabular}
%\end{table}

\begin{figure}[h!]
    \centering
    \includegraphics[width=.50\textwidth]{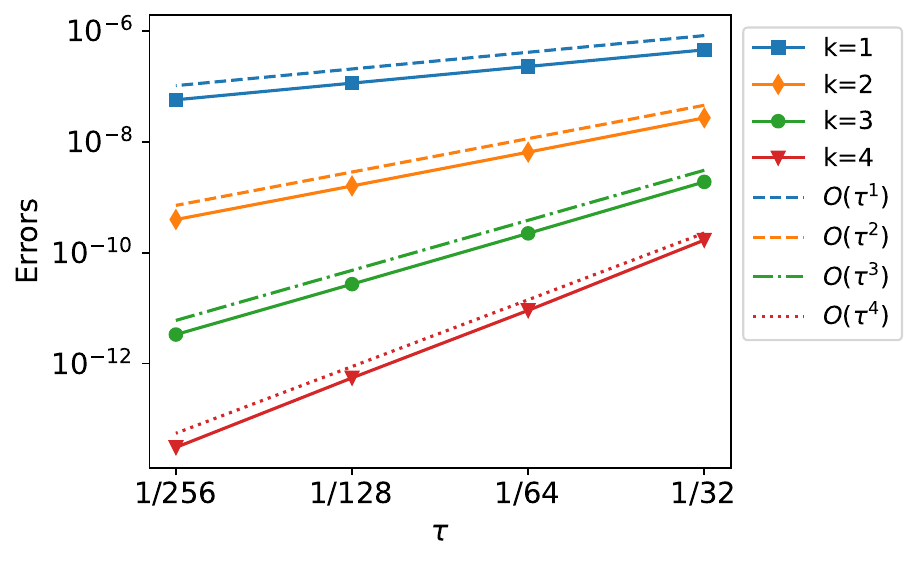}
    \caption{Example \ref{example62}: Errors of time discretization with $m=1$ and $ r=3$.}
    \label{fig:exa53_time}
\end{figure}

\begin{table}[h!]
\centering\scriptsize
\caption{Example \ref{example62}: Comparison of $\alpha=0.6$ and $\alpha=1$, with $m=0$.}\label{table:exa53_space_m0}
\vskip-5pt
\begin{tabular}{ccccccccccccc}
    \toprule
             &\multicolumn{6}{c}{$\alpha=0.6$}
             &\multicolumn{6}{c}{$\alpha=1$}
    \\\cmidrule(lr){2-7}\cmidrule(lr){8-13}
             &\multicolumn{2}{c}{$r=1$}
             &\multicolumn{2}{c}{$r=2$}
             &\multicolumn{2}{c}{$r=3$}
             &\multicolumn{2}{c}{$r=1$}
             &\multicolumn{2}{c}{$r=2$}
             &\multicolumn{2}{c}{$r=3$}
    \\\cmidrule(lr){2-3}\cmidrule(lr){4-5}\cmidrule(lr){6-7}\cmidrule(lr){8-9}\cmidrule(lr){10-11}\cmidrule(lr){12-13}
        $h_j$ & $E_{h_j}$ &  conv. & $E_{h_j}$ &  conv. & $E_{h_j}$ &  conv. & $E_{h_j}$ &  conv. & $E_{h_j}$ &  conv. & $E_{h_j}$ &  conv.\\
    \midrule
         1/32 &  3.35e-05 &    -   &  4.40e-06 &    -   &  1.42e-06 &    -   &  3.90e-11 &    -   &  2.35e-14 &    -   &  1.30e-16 &    -  \\
         1/64 &  8.99e-06 &   1.90 &  1.10e-06 &   2.00 &  3.56e-07 &   2.00 &  9.85e-12 &   1.98 &  2.71e-15 &   3.12 &  8.14e-18 &   4.00\\
        1/128 &  2.39e-06 &   1.91 &  2.75e-07 &   2.00 &  8.91e-08 &   2.00 &  2.47e-12 &   2.00 &  3.31e-16 &   3.03 &  5.09e-19 &   4.00\\
        1/256 &  6.31e-07 &   1.92 &  6.88e-08 &   2.00 &  2.23e-08 &   2.00 &  6.18e-13 &   2.00 &  4.12e-17 &   3.01 &  3.18e-20 &   4.00\\
    \cmidrule{1-13}
    \multicolumn{2}{c}{theor. conv.}
                          &   2.00 &           &   2.00 &           &   2.00 &           &   2.00 &           &   3.00 &           &   4.00\\
    \bottomrule
\end{tabular}
\end{table}
\begin{table}[h!]
\centering\scriptsize
\caption{Example \ref{example62}: Improved spatial convergence by using \eqref{eqn:fully} with $m=1$.}
\label{table:exa53_space_m1}
\vskip-5pt
\begin{tabular}{ccccccccccccc}
    \toprule
    \multirow{3}{*}{$\alpha=0.6$}
             &\multicolumn{6}{c}{regular part}
             &\multicolumn{6}{c}{singular part}
    \\\cmidrule(lr){2-7}\cmidrule(lr){8-13}
             &\multicolumn{2}{c}{$r=1$}
             &\multicolumn{2}{c}{$r=2$}
             &\multicolumn{2}{c}{$r=3$}
             &\multicolumn{2}{c}{$r=1$}
             &\multicolumn{2}{c}{$r=2$}
             &\multicolumn{2}{c}{$r=3$}
    \\\cmidrule(lr){2-3}\cmidrule(lr){4-5}\cmidrule(lr){6-7}\cmidrule(lr){8-9}\cmidrule(lr){10-11}\cmidrule(lr){12-13}
        $h_j$ & $E_{h_j}$ &  conv. & $E_{h_j}$ &  conv. & $E_{h_j}$ &  conv. & $E_{h_j}$ &  conv. & $E_{h_j}$ &  conv. & $E_{h_j}$ &  conv.\\
    \midrule
         1/32 &  5.66e-07 &    -   &  3.95e-09 &    -   &  5.03e-11 &    -   &  3.33e-05 &    -   &  4.07e-07 &    -   &  1.15e-08 &    -  \\
         1/64 &  1.42e-07 &   2.00 &  4.96e-10 &   2.99 &  3.41e-12 &   3.88 &  8.97e-06 &   1.89 &  5.26e-08 &   2.95 &  7.82e-10 &   3.87\\
        1/128 &  3.55e-08 &   2.00 &  6.21e-11 &   3.00 &  2.29e-13 &   3.90 &  2.38e-06 &   1.91 &  6.46e-09 &   3.03 &  4.86e-11 &   4.01\\
        1/256 &  8.88e-09 &   2.00 &  7.77e-12 &   3.00 &  1.52e-14 &   3.91 &  5.16e-07 &   2.21 &  8.57e-10 &   2.92 &  3.04e-12 &   4.00\\
    \cmidrule{1-13}
    \multicolumn{2}{c}{theor. conv.}
                          &   2.00 &           &   3.00 &           &   4.00 &           &   2.00 &           &   3.00 &           &   4.00\\
    \bottomrule
\end{tabular}
\end{table}\bigskip

%\begin{figure}[h!]
%    \centering
%    \subfigure[$\alpha=1.0, m=0$                ]{ \includegraphics[width=.32\textwidth]{{exa53_space_cauchy_k=4_m=0_diff_alpha1}.pdf}}
%    \subfigure[$\alpha=0.6, m=1$ (regular part) ]{ \includegraphics[width=.32\textwidth]{{exa53_space_cauchy_k=4_m=1_alpha0.6_reg}.pdf}} \\
%    \subfigure[$\alpha=0.6, m=0$                ]{ \includegraphics[width=.32\textwidth]{{exa53_space_cauchy_k=4_m=0_diff_alpha0.6}.pdf}}
%    \subfigure[$\alpha=0.6, m=1$ (singular part)]{ \includegraphics[width=.32\textwidth]{{exa53_space_cauchy_k=4_m=1_alpha0.6_sig}.pdf}}
%    \caption{Example \ref{example62}: spatial error for different $m$ and $r$.}
%    \label{fig:exa53_space}
%\end{figure}

%\FloatBarrier

\end{example}

\section{Conclusions}

We have constructed a new splitting of the solution to the subdiffusion equation, which allows us to develop high-order finite element approximations in
case of nonsmooth initial data. In this method, the solution is split into a time-dependent smooth part plus a time-independent nonsmooth part. We have developed high-order spatial and time discretizations to approximate the smooth part of the solution, and proved that the proposed fully discrete finite element method approximates the regular part of the solution to high-order accuracy for nonsmooth initial data in $L^2(\Omega)$.
Moreover, we have illustrated how to approximate the time-independent nonsmooth part through several examples of initial data, including piecewise smooth initial data, Dirac--Delta point source, and Dirac measure concentrated on an interface. More generally, the time-independent nonsmooth part can be approximated by using smaller mesh size without increasing the overall computational cost significantly. This is possible as the nonsmooth part is time-independent and therefore avoids the time-stepping procedure.
%The argument could be directly extended to subdiffusion equations with nonsmooth source data.
%We have proved that the fully discrete finite element method proposed in this paper approximates the regular part of the solution to high-order accuracy for nonsmooth initial data in $L^2(\Omega)$.
We have also illustrated the effectiveness of the proposed method through several numerical examples.
% are presented to support the theoretical analysis and to illustrate the performance of the proposed high-order splitting finite element methods.

\section*{Appendix: On the triangulation satisfying Assumption \ref{ass:mesh}}$\,$
\renewcommand{\theequation}{A.\arabic{equation}}
\renewcommand{\thelemma}{A.\arabic{lemma}}
\renewcommand{\theproposition}{A.\arabic{proposition}}
\setcounter{equation}{0}

In this Appendix we show that the graded mesh defined in \eqref{mesh-condition}, for a two-dimensional polygonal domain $\Omega$, satisfies Assumption \ref{ass:mesh}.

In terms of the notation introduced in \eqref{mesh-condition} and the subsequent text, we divide the domain $D_0=\{x\in\Omega:|x-x_0|<d_0\}$ into
$D_0= D_*\cup (\cup_{j=1}^JD_j)$ with
$$
D_j:=\{x\in\Omega: 2^{-j-1}d_0\le |x-x_0|<2^{-j}d_0\}
\quad\mbox{and}\quad
D_*:=\{x\in\Omega: |x-x_0|\le 2^{-J}d_0=h_*\}.
$$
Let $d_j=2^{-j}d_0$ and $h_j= \max\limits_{x\in D_j} h(x) \sim d_j^{1-\gamma}h$, and denote by
$$D_j':=\{x\in\Omega: 2^{-j-2}d_0\le |x-x_0|<2^{-j+1}d_0\} $$ a neighborhood of $D_j$.
Then the following lemma provides a regularity estimate near the corner.

%\red{specify some notations, $d_j$, $D_j'$, $D$, $f$}

\begin{lemma}\label{LemmaA}
If $A=-\Delta$ and $v\in \dot H^{2r+2}(\Omega)$ for some $r\ge 0$, then
\begin{align*}
\begin{aligned}
\|v\|_{H^{s}(D_j)}
\le
C d_j^{1-s+\min(1,\pi/\theta)}  \|v\|_{\dot H^{2r+2}(\Omega)}
\quad\mbox{for}\,\,\,
0\le s\le 2r+2 .
\end{aligned}
\end{align*}
\end{lemma}
\begin{proof}
For $v\in \dot H^{2r+2}(\Omega)$, the following weighted regularity result is known (for example, see \cite[Proof of Lemma 5.1, inequality (5.8)]{Li-2022}):
\begin{align}\label{local-energy}
|v|_{H^{k+1}(D_j)}
&\leq C \sum_{i=0}^{k-1} d_j^{-i}\|\Delta v\|_{H^{k-1-i}(D_j')}
+C d_j^{-k} \|\nabla v\|_{L^2(D_j')} \quad 1\le k\le 2r+1 .
\end{align}
By using the singular expansions and local energy estimate, i.e.,
$$ v|_{D_0} \in O( |x-x_0|^{\pi/\theta}) + H^2(D_0)
\quad\mbox{and}\quad
\|\nabla v\|_{L^2(D_j)}^2
\le Cd_j^{-2}\| v\|_{L^2(D_j')}^2
+Cd_j^2\|f\|_{L^2(D_j')}^2 ,
$$
we further obtain
\begin{align*}
|v|_{H^{k+1}(D_j)}  &\leq C \sum_{i=0}^{k-1} d_j^{-i}\|\Delta v\|_{H^{k-1-i}(D_j')}
+C d_j^{-k+\min(1,\pi/\theta)} \|f\|_{L^{2}(\Omega)}
\quad 1\le k\le 2r+1 .
\end{align*}
%\begin{align*}
%\|v\|_{H^{r+1}(D_j)} \leq C d_j^{-r+\min(1,\pi/\theta_j) } \|f\|_{H^{r-1}(\Omega)}
%\quad\mbox{with}\,\,\, d_j=2^{-j}d_0 .
%\end{align*}
%$$\|f\|_{L^2(D_j')} \le Cd_j^{\min(1,\pi/\theta)+1} $$
If $\|\Delta v\|_{H^{s}(D_j')} \le Cd_j^{-s+\min(1,\pi/\theta)}$ for $0\le s\le 2r$, then the inequality above yields
\begin{align*}
\begin{aligned}
|v|_{H^{s}(D_j)}
&\leq  C d_j^{1-s+\min(1,\pi/\theta)} ,\quad 2\le s\le 2r+2 .
\end{aligned}
\end{align*}
By using the singular expansions and local energy estimate in \eqref{local-energy}, we find that the inequality above also holds for $s=0,1$.
Since $d_j\le C$, it follows that $\| v\|_{H^{s}(D_j')} \le Cd_j^{-s+\min(1,\pi/\theta)}$ for $0\le s\le 2r+2$. This is the same estimate as that for $\Delta v$, but the range of $s$ increases by $2$. Therefore, for any function $f\in L^2(\Omega)$, by picking up the regularity from $\Delta^{-l}f$ to $\Delta^{-l-1}f$ for $l=0,1,\dots,r$, we obtain
\begin{align*}
\begin{aligned}
|\Delta^{-r-1}f|_{H^{s}(D_j)}
\le
C d_j^{1-s+\min(1,\pi/\theta)} \|f\|_{L^{2}(\Omega)}
\quad\mbox{for}\,\,\, \quad 0\le s\le 2r+2 .
\end{aligned}
\end{align*}
Replacing $\Delta^{-r-1}f$ by $v$ in the estimate above, we obtain the result of Lemma \ref{LemmaA}.
\end{proof}

Then the following proposition confirms the approximation properties \eqref{eqn:int-err}--\eqref{Error-solution-f}.
\begin{proposition} \label{proposition:FEM}
Let $A=-\Delta$.
If the diameter of triangles satisfies condition \eqref{mesh-condition} near every corner of the domain,
then \eqref{eqn:int-err}--\eqref{Error-solution-f} holds.
\end{proposition}
\begin{proof}
Let $v\in \dH{2r+2}$ and $s=2r+1$, with $0\le r\le m$. We consider the following decomposition:
\begin{align}\label{v-Ihv}
\|v-I_hv\|_{H^1(D_0)}^2
&= \|v-I_hv\|_{H^1(D_*)}^2+\sum_{j=1}^J \|v-I_hv\|_{H^1(D_j)}^2 \notag\\
&\le
Ch_*^{2\min(1,\pi/\theta)} \|\Delta v\|_{L^2(\Omega)}^2 + C\sum_{j=1}^Jh_j^{2s}\|v\|_{H^{s+1}(D_j')}^2 ,
\end{align}
and
\begin{align}\label{v-Ihv-L2}
\|v-I_hv\|_{L^2(D_0)}^2
&= \|v-I_hv\|_{L^2(D_*)}^2+\sum_{j=1}^J \|v-I_hv\|_{L^2(D_j)}^2 \notag\\
&\le
Ch_*^{2+2\min(1,\pi/\theta)} \|\Delta v\|_{L^2(\Omega)}^2 + C\sum_{j=1}^Jh_j^{2s+2}\|v\|_{H^{s+1}(D_j')}^2 ,
\end{align}
where $h_j=\max_{x\in D_j} \hbar(x)$, and we have used the following result:
\begin{align}\label{H1-local}
\|v-I_hv\|_{H^1(D_*)}
\le Ch_*^{\min(1,\pi/\theta)} \|\Delta v\|_{L^2(\Omega)} .
\end{align}
In the case $\pi/\theta<1$, this result follows from
$$
\|v-I_hv\|_{H^1(D_*)}
\le Ch_*^{\min(1,\pi/\theta)} \|v\|_{B^{\min(1,\pi/\theta)}_{2,\infty}(\Omega)}
\le Ch_*^{\min(1,\pi/\theta)} \|\Delta v\|_{L^2(\Omega)} ,
$$
where $B^{\min(1,\pi/\theta)}_{2,\infty}(\Omega)$ is the Besov space.
The last inequality is a consequence of the singularity expansion
$$v|_{D_0} = c|x-x_0|^{\pi/\theta}\sin({\rm arg}(x-x_0)) + w~~ \text{for some} ~ c\in\R~\text{and}~ w\in H^2(D_0), $$
where $|c| \le \| \Delta v \|_{L^2\II}$.
In the case $\pi/\theta>1$, \eqref{H1-local} follows from the standard estimate for the Lagrange interpolation, i.e.,
$$
\|v-I_hv\|_{H^1(D_*)}
\le Ch_* \|v\|_{H^2(\Omega)}
\le Ch_* \|\Delta v\|_{L^2(\Omega)} .
$$

By substituting the result of Lemma \ref{LemmaA} into \eqref{v-Ihv} and using the condition $\gamma_j<\frac{\min(1,\pi/\theta_j)}{r} $, we obtain
\begin{align*}
\|v-I_hv\|_{H^1(D_0)}^2 &\le
Ch_*^{2\min(1,\pi/\theta)} \|v\|_{\dot H^{2m+2}(\Omega)}^2
+
\sum_{j} Ch_j^{2s} d_j^{-2s+2\min(1,\pi/\theta_j) } \|v\|_{\dot H^{2r+2}(\Omega)}^2\\
&\le
Ch^{2s} \|v\|_{\dot H^{2r+2}(\Omega)}^2
+\sum_{j} C d_j^{(1-\gamma_j)2s-2s+2\min(1,\pi/\theta_j) } h^{2s} \|v\|_{\dot H^{2r+2}(\Omega)}^2\\
&\le
Ch^{2s}\|v\|_{\dot H^{2r+2}(\Omega)}^2+\sum_{j} C d_j^{2s\big(\frac{\min(1,\pi/\theta_j)}{s} - \gamma_j \big) }
h^{2s} \|v\|_{\dot H^{2r+2}(\Omega)}^2 \\
&\le Ch^{2s} \|v\|_{\dot H^{2r+2}(\Omega)}^2 .
\end{align*}
This proves that, by substituting $s=2r+1$ into the inequality above,
\begin{align*}
\|v-I_hv\|_{H^1(\Omega)}&\le Ch^{2r+1} \|v\|_{\dot H^{2r+2}(\Omega)} .
\end{align*}
Similarly, by substituting the result of Lemma \ref{LemmaA} into \eqref{v-Ihv-L2}, we obtain
%\begin{align*}
%\|v-I_hv\|_{L^2(D)}^2
%&\le
%Ch_*^{2+2\min(1,\pi/\theta)} \|v\|_{\dot H^{2m+2}(\Omega)}^2
%+
%\sum_{j} Ch_j^{2r+2} d_j^{-2r+2\min(1,\pi/\theta_j) } \|v\|_{\dot H^{2m+2}(\Omega)}^2\\
%&\le
%Ch^{2r+2} \|v\|_{\dot H^{2m+2}(\Omega)}^2
%+\sum_{j} C d_j^{(1-\gamma_j)(2r+2)-2r+2\min(1,\pi/\theta_j) } h^{2r+2} \|v\|_{\dot H^{2m+2}(\Omega)}^2\\
%&\le
%Ch^{2r+2}\|v\|_{\dot H^{2m+2}(\Omega)}^2+\sum_{j} C d_j^{2r\big(\frac{\min(1,\pi/\theta_j)}{r} - \gamma_j \big) +2(1-\gamma_j)}
%h^{2r+2} \|v\|_{\dot H^{2m+2}(\Omega)}^2 \\
%&\le Ch^{2r+2} \|v\|_{\dot H^{2m+2}(\Omega)}^2 .
%\end{align*}
%This proves that
\begin{align*}
\|v-I_hv\|_{L^2(\Omega)}&\le Ch^{r+1} \|v\|_{\dot H^{2r+2}(\Omega)} .
\end{align*}
This proves the desired estimate in \eqref{eqn:int-err}.

By the optimal $H^1$-norm approximation property of the Ritz projection,
we have
\begin{align*}
\|v-R_hv\|_{H^1(\Omega)}
\le C\|v-I_hv\|_{H^1(\Omega)}
\le Ch^{2r+1} \|v\|_{\dot H^{2r+2}(\Omega)} .
\end{align*}
By a standard duality argument, we obtain
%\begin{align*}
%\|v-v_h\|_{L^2(\Omega)}^2
%&=(\nabla (v-v_h),\nabla (\phi-I_h\phi)) \\
%&\le Ch\|v-v_h\|_{H^1(\Omega)} \|v-v_h\|_{L^2(\Omega)} \\
%\end{align*}
\begin{align*}
\|v-R_hv\|_{L^2(\Omega)} \le Ch\|v-R_hv\|_{H^1(\Omega)} \le Ch^{2r+2} \|v\|_{\dot H^{2r+2}(\Omega)} .
\end{align*}
This proves the desired estimate in \eqref{Error-solution-f}.
\end{proof}
%\end{example}

%The following approximation properties of $R_h$ and $P_h$ are well known:
%\begin{align}
%  \|P_h\psi-\psi\|_{L^2(\Omega)}+h\|\nabla(P_h\psi-\psi)\|_{L^2(\Omega)}& \leq ch^r\|\psi\|_{H^r(\Omega)},\quad \forall v\in H^{r}(\Omega)\cap H_0^1\II. \label{eqn:err-Ph}.
%  %\|R_h\psi-\psi\|_{L^2(\Omega)}+h\|\nabla(R_h\psi-\psi)\|_{L^2(\Omega)}& \leq ch^q\|\psi\|_{H^q(\Omega)}\quad \forall\psi\in \dH q, q=1,2.\label{eqn:err-Rh}
%\end{align}

\begin{lemma} \label{lemma:err-I_h-Lyinfty}
If the mesh size satisfies condition \eqref{mesh-condition}, then the following estimate holds:
\begin{align}
 \| v - I_h v \|_{L^\infty(\Omega)} &\leq ch^{2r+1}\| v\|_{\dH{2r+2}}
\label{eqn:err-I_h-Lyinfty}
\quad\mbox{for}\,\,\, 0\le r\le m.
\end{align}
\end{lemma}

\begin{proof}
The basic $L^\infty$ estimates of the Lagrange interpolation says that
\begin{align*}
\|v-I_hv\|_{L^\infty(D_j)}
&\le c h_j^{2r+1} \|v\|_{H^{2r+2}(D_j')} .
\end{align*}
Since $\|v\|_{H^{2r+2}(D_j')} \le cd_j^{-2r-1+\min(1,\pi/\theta)}\|v\|_{\dot H^{2r+2}(\Omega)}$, it follows that
\begin{align*}
\|v-I_hv\|_{L^\infty(D_j)}
&\le cd_j^{-2r-1+\min(1,\pi/\theta)} h_j^{2r+1} \|v\|_{\dot H^{2r+2}(\Omega)} \\
&\le cd_j^{-2r-1+\min(1,\pi/\theta)} d_j^{(1-\gamma)(2r+1)}h^{2r+1} \|v\|_{\dot H^{2r+2}(\Omega)} \\
&\le cd_j^{\min(1,\pi/\theta)-(2r+1)\gamma} h^{2r+1} \|v\|_{\dot H^{2r+2}(\Omega)} .
\end{align*}
Then \eqref{eqn:err-I_h-Lyinfty} follows from the condition $\gamma<\min(1,\pi/\theta)/(2m+1)\le \min(1,\pi/\theta)/(2r+1)$ in \eqref{mesh-condition}.
\end{proof}

\section*{Acknowledgements}
The research of B. Li is partially supported by Hong Kong Research Grants Council  (GRF Project No. 15300817) and an internal grant of The Hong Kong Polytechnic University (Project ID: P0031035, Work Programme: ZZKQ). 
The research of Z. Zhou is partially supported by Hong Kong Research Grants Council (Project No. 	25300818) and an internal grant of The Hong Kong Polytechnic University (Project ID: P0031041, Work Programme: ZZKS).

\bibliographystyle{abbrv}

\end{document}